\documentclass[a4paper,11pt]{article}
\usepackage[utf8]{inputenc}
\usepackage[T1]{fontenc}
\usepackage[margin=3cm]{geometry}

\usepackage[scaled=.95]{helvet}
\usepackage{courier}
\usepackage{amsmath,amsfonts,amssymb,amsthm}
\usepackage{graphicx}
\usepackage{natbib}
\usepackage{color}
\usepackage{comment}

\newcommand{\egaldef}{\mathrel{:=}}
\newcommand*{\CF}{\mathcal{F}}
\newcommand*{\E}{\mathbb{E}}
\newcommand*{\EE}[1]{\E\left[#1\right]}

\newcommand*{\xR}{\mathbb{R}}

\newcommand*{\scal}[1]{\left\langle #1 \right\rangle} % produit scalaire
\newcommand*{\nrm}[1]{\left\| #1 \right\|}            % norme
\newcommand*{\Var}{\mathbf{Var}}                      % variance
\newtheorem{assmptn}{Assumption}
\newtheorem{thm}{Theorem}[section]
\newtheorem{cor}[thm]{Corollary}
\newtheorem{lem}[thm]{Lemma}
\newtheorem{prop}[thm]{Proposition}

\begin{document}

\title{Fast Estimation of the Median Covariation Matrix with Application to Online Robust Principal Components Analysis}
%Robust Principal Components Analysis based on the Median Covariation Matrix}

\author{Herv\'e \textsc{Cardot}, Antoine \textsc{Godichon-Baggioni} \\ Institut de Math\'ematiques de Bourgogne, \\ Universit\'e de Bourgogne Franche-Comt\'e, \\
9, rue Alain Savary, 21078 Dijon, France} 
\maketitle

%\tableofcontents

\begin{abstract}
The geometric median covariation matrix  is a robust multivariate indicator of dispersion which can be extended without any difficulty to functional data. 
We define estimators, based on recursive algorithms, that can be simply updated at each new observation  and are able to deal rapidly with  large samples of high dimensional data without being obliged to store all the data in memory.  Asymptotic convergence properties of the  recursive  algorithms are studied under weak conditions. The computation of the principal components can also be performed online  and this approach can  be useful for online outlier detection.
A simulation study clearly shows that this robust indicator is a competitive alternative to  minimum covariance determinant  when the dimension of the data is small and  robust  principal components analysis based on projection pursuit and spherical projections  for high dimension data.  An illustration on a large sample and high dimensional dataset consisting of individual TV audiences measured at a minute scale over a period of 24 hours confirms  the interest of considering the robust principal components analysis based on the median covariation matrix. All studied algorithms are available in the R package \texttt{Gmedian} on CRAN.
\end{abstract}


\section{Introduction}

Principal Components Analysis is one of the most useful statistical tool to extract information by reducing the dimension when one has to analyze large samples of multivariate or  functional data (see {\it e.g.}  \cite{Jolliffe2002} or \cite{RamsaySilverman2005}). When both the dimension and the sample size are large, outlying observations may be  difficult  to detect automatically. Principal components, which are derived from the spectral analysis of the covariance matrix, can be very sensitive to outliers (see \cite{DGK1981}) and many robust procedures for  principal components analysis have been considered in the literature  (see \cite{HRVA2008}, \cite{HubR2009} and \cite{MR2238141}).
 
%Different techniques have been studied in the literature to estimate robustly the covariance matrix :
The most popular approaches are probably  the minimum covariance determinant  estimator (see \cite{RvD99}) and the robust projection pursuit  (see \cite{CR-G2005} and \cite{CFO2007}). Robust PCA based on projection pursuit has been extended to deal with functional data in \cite{HyndmanUllah2007} and \cite{BBTW2011}. Adopting another point of view, robust modifications of the covariance matrix, based on projection of the data onto the unit sphere, have been proposed in  \cite{LMSTZC1999} (see also  \cite{Ger08} and \cite{TKO2012}). 

We consider in this work another robust way of measuring association between variables, that can be extended directly to functional data. It is based on the notion of  median covariation matrix (MCM)  which is defined as the minimizer of an expected loss criterion based on the Hilbert-Schmidt norm (see \cite{KrausPanaretos2012} for a first definition in a more general $M$-estimation setting).  It can be seen as a geometric median (see \cite{Kem87} or \cite{MNO2010}) in  the particular Hilbert  spaces of square matrices (or operators for functional data) equipped with the Frobenius (or Hilbert-Schmidt) norm. 
The MCM  is non negative and unique  under weak conditions. As shown in \cite{KrausPanaretos2012} it also has the same eigenspace as the usual covariance matrix when the distribution of the data is symmetric and the second order moment is finite.  
Being a spatial median in a particular Hilbert space of matrices, the MCM is also a robust indicator of central location, among the covariance matrices, which  has a 50 \% breakdown point (see  \cite{Kem87} or \cite{MR2238141})  as well as a bounded gross sensitivity error (see \cite{CCZ11}). 

The aim of this  work is twofold.  It provides efficient recursive estimation algorithms of the MCM that are able to deal with large samples of high dimensional data. By this recursive property, these algorithms can naturally deal with data that are observed sequentially and provide a natural update of the estimators at each new observation. Another advantage compared to classical approaches is that  such recursive algorithms will not require to store all the data. Secondly, this work also aims at highlighting  the interest of considering the median covariation matrix to perform principal components analysis  of high dimensional contaminated data.

Different algorithms can be considered to get effective estimators of the MCM.  
When the dimension of the data is not too high and the sample size is not too large, Weiszfeld's algorithm (see \cite{Weiszfeld1937} and \cite{VZ00})  can be directly used to estimate effectively both the geometric median and the median covariation matrix.
When both the dimension and the sample size are large this static algorithm which requires to store all the data may be inappropriate and ineffective.   We show how the algorithm developed by \cite{CCZ11} for the geometric median in Hilbert spaces can be adapted to estimate recursively and simultaneously the median as well as the median covariation matrix. Then an averaging step (\cite{PolyakJud92}) of the two initial recursive estimators of the median and the MCM  permits to improve the accuracy of the initial stochastic gradient algorithms.  A simple modification of the stochastic gradient algorithm is proposed in order to ensure that the median covariance estimator is non negative. We also explain how the eigenelements of the estimator of the MCM can be updated online without being obliged to perform a new spectral decomposition at each new observation. 
 
The paper is organized as follows. The median covariation matrix as well as the recursive estimators are defined in Section 2. 
In Section 3, almost sure and quadratic mean consistency results are given for variables taking values in general separable Hilbert spaces.   The proofs, which are based on  new  induction steps compared to \cite{CCZ11}, allow to get better convergence rates in quadratic mean even if this new framework is much more complicated because two averaged non linear  algorithms are running simultaneously.   One can also note that the techniques generally employed to deal with  two time scale Robbins Monro algorithms (see \cite{MR2260078} for the multivariate case) require assumptions on the rest of the Taylor expansion and the finite dimension of the data that are too restrictive in our framework. 
In Section 4, a  comparison with some classic robust PCA techniques  is made on simulated data. The interest of considering the MCM is also highlighted  on the analysis of individual TV audiences, a large sample of high dimensional data which, because of its dimension, can not be analyzed in a reasonable time with classical robust PCA approaches.  The main parts of the proofs are described in Section 5. Perspectives for future research are discussed in Section 6. Some technical parts of the proofs as well as a description of Weiszfeld's algorithm in our context are gathered in  an Appendix. %a supplementary file.

%\medskip

%Citer  \cite{Minsker2014} and \cite{ChaCha2014} for recent work on the geometric median in Hilbert spaces.
%Citer \cite{Cuevas2014} for a recent review on functional data analysis.

\section{Population point of view and recursive estimators}
Let $H$ be a separable Hilbert space (for example $H = \xR^d$ or $H = L^2(I)$, for some closed interval $I \subset \mathbb{R}$). We  denote by $\langle .,.\rangle$  its inner product and by  $\nrm{\cdot}$ the associated norm. 

We consider a random variable  $X$ that  takes values in $H$ and define its center $m \in H$ as follows:
\begin{align}
m  & \egaldef  \arg \min_{u \in H} \EE{\nrm{X - u} - \nrm{X}} .
\label{defmed}
\end{align}
The solution $m \in H$ is often called the geometric median of $X$. It is uniquely defined under broad assumptions on the distribution of $X$ (see \cite{Kem87}) which can be expressed as follows.

 \begin{assmptn}\label{eq:supportCdtnmed}
 There exist two linearly independent unit vectors $(u_1,u_2) \in H^2$, such that
\[
    \Var ( \scal{u,X} ) > 0, \quad \mbox{for }u \in \{u_1,u_2\} .
\]
\end{assmptn}

If the distribution of $X-m$ is symmetric around zero and if $X$ admits a first moment that is finite then the geometric median is equal to the expectation of $X$,  $m = \EE{X}$. 
Note however that the general definition (\ref{defmed}) does not require to assume that the first order moment of $\nrm{X}$ is finite since $| \EE{ \nrm{X-u} - \nrm{X}} | \leq \nrm{u}$. 

\subsection{The (geometric) median covariation matrix (MCM)}

We now consider  the special vector space, denoted by $\mathcal{S}(H)$, of $d \times d$ matrices when $H= \mathbb{R}^d$, or for general separable Hilbert spaces $H$, the vector space of linear operators mapping $H \to H$. Denoting by $\{e_j, j \in J \}$ an orthonormal basis in $H$, the vector space $\mathcal{S}(H)$ equipped with the following inner product:
\begin{align}
\langle A, B \rangle_F &= \sum_{j \in J} \langle A e_j, B e_j \rangle
\end{align}
 is also a separable Hilbert space. In $\mathcal{S}(\mathbb{R}^d)$, we have equivalently
 \begin{align}
\langle A, B \rangle_F &= \mbox{tr} \left( A^T B \right),
\end{align}
where $A^T$ is the transpose matrix of $A$. 
 The induced norm is the well known Frobenius norm (also called Hilbert-Schmidt norm) and is denoted by $\nrm{.}_F .$
 %= \sqrt{  \mbox{tr} \left( \mathbf{A}^T \mathbf{A} \right)}$.
%Note that, if $e_1, \ldots, $ is an orthonormal basis in $H$,
%\begin{align*}
%\nrm{\mathbf{A}}_F^2 &=  \sum_{\ell } \nrm{\mathbf{A} e_\ell}^2
%\end{align*}

When $X$ has finite second order moments, with expectation $\EE{X}=\mu$,  the covariance matrix of $X$, $\EE{(X-\mu)(X-\mu)^T}$ can be defined as the minimum argument, over all the elements belonging to $\mathcal{S}(H)$, of the functional $G_{\mu,2} : \mathcal{S}(H) \to \mathbb{R}$,
\[
G_{\mu,2}(\Gamma) = \EE{ \nrm{(X-\mu)(X-\mu)^T - \Gamma}^2_F - \nrm{(X-\mu)(X-\mu)^T}_F^2}. 
\]  
Note that in general Hilbert spaces with inner product $\langle ., . \rangle$, operator $(X-\mu)(X-\mu)^T$ should be understood as the operator $u \in H \mapsto \langle u, X-\mu \rangle (X-\mu)$.
The MCM is obtained by removing the squares in previous function in order to get a more robust indicator of "covariation". 
For $\alpha \in H$, define $G_\alpha : \mathcal{S}(H) \to \mathbb{R}$ by
\begin{align}
G_\alpha (V) &:= \EE{ \nrm{(X-\alpha)(X-\alpha)^T - V}_F - \nrm{(X-\alpha)(X-\alpha)^T}_F} . 
\label{def:popriskcov}
\end{align}
The median covariation matrix, denoted by  $\Gamma_m$, is defined as the minimizer of $G_m(V)$ over all elements $V \in \mathcal{S}(H)$.
The second term at the right-hand side of (\ref{def:popriskcov}) prevents from having to introduce hypotheses on the existence of the moments of $X$.
Introducing the random variable $Y := (X- m)(X- m)^T$ that takes values in $\mathcal{S}(H)$, the MCM is unique provided that the support of $Y$ is not concentrated on a line  and Assumption 1 can be rephrased as follows in $\mathcal{S}(H)$,
\begin{assmptn}\label{eq:supportCdtnCov}
 There exist two linearly independent unit vectors $(V_1,V_2) \in \mathcal{S}(H)^2$, such that
\[
    \Var ( \scal{V,Y}_F ) > 0, \quad \mbox{for }V \in \{V_1,V_2\} .
\]
\end{assmptn}

We can remark that Assumption \ref{eq:supportCdtnmed} and Assumption \ref{eq:supportCdtnCov} are strongly connected. 
Indeed, if Assumption \ref{eq:supportCdtnmed} holds, then $\Var ( \scal{u,X} ) > 0$ for $ u \in \{u_1, u_2\}$. Consider the rank one matrices $V_1 = u_1u_1^T$ and $V_2 = u_2 u_2^T$, we have $\scal{V_1,Y}_F = \langle u_1, X-m \rangle^2$ which has a strictly positive variance when the distribution of $X$ has no atom. More generally $ \Var(\scal{V_1,Y}_F) >0$ unless there is a scalar $a >0$ such that $\mathbb{P}\left[ \langle u_1, X-m \rangle = a\right] = \mathbb{P}\left[ \langle u_1, X-m \rangle = -a\right] = \frac{1}{2}$ (assuming also that $\mathbb{P}\left[  X-m = 0\right] = 0$).

Furthermore it can be deduced easily that the MCM, which is a geometric median in the particular Hilbert spaces of Hilbert-Schmidt operators,  is a robust indicator with a 50\% breakdown point (see \cite{Kem87}) and a bounded sensitive gross error (see \cite{CCZ11}).

We also assume that
\begin{assmptn}\label{eq:invMomentCov}
 There is a constant $C$ such that for all $h \in H$ and all $V \in \mathcal{S}(H)$
 \begin{align*}
 (a) &: \quad    \EE{ \nrm{(X- h)(X- h)^T - V}^{-1}_F } \leq C. \\
 (b) &: \quad    \EE{ \nrm{(X- h)(X- h)^T - V}^{-2}_F } \leq C. \\
\end{align*}
\end{assmptn}
This assumption implicitly forces the distribution of $(X- h)(X- h)^T$ to have no atoms. It is more "likely" to be satisfied when the dimension $d$ of the data is large (see \cite{Cha92} and  \cite{CCZ11} for a discussion). Note that it could be weakened as  in \cite{CCZ11} by allowing points, necessarily different from the MCM $\Gamma_m$,  to have strictly positive masses. 
Considering the particular case $V=0$, Assumption~\ref{eq:invMomentCov}(a) implies that for all $h \in H$,
\begin{align}
\EE{ \frac{1}{\nrm{X- h}^{2}} } \leq C,
\label{cond:mominv2x}
\end{align}
and this is not restrictive when the dimension  $d$ of $H$ is equal or larger than 3. 

Under  Assumption~\ref{eq:invMomentCov}(a),  the functional $G_h$ is twice Fr\'echet differentiable, with gradient
\begin{align}
\nabla G_h (V) &= - \EE{ \frac{(X- h)(X- h)^T - V}{\nrm{(X- h)(X- h)^T - V}_F}}.
\label{def:gradV}
\end{align}
and Hessian operator, $ \nabla _h^2 G(V) : \mathcal{S}(H) \to \mathcal{S}(H)$,
\begin{align}
\nabla _h^2G (V) &=  \EE{ \frac{1}{\nrm{Y(h) - V}_F}\left( I_{S(H)} - \frac{(Y(h) - V) \otimes_F  (Y(h) - V) }{{\nrm{Y(h) - V}_F}^2} \right)}.
\label{def:HeV}
\end{align}
where $Y(h) = (X-h)(X-h)^T$, $I_{S(H)}$ is the identity operator on $\mathcal{S}(H)$  and $A \otimes_F B (V) = \langle A, V \rangle_F B$ for any elements $A, B$ and $V$ belonging to $\mathcal{S}(H)$.

Furthermore, $\Gamma_m$ is also defined as the unique zero of the non linear equation:
\begin{align}
\nabla G_m (\Gamma_m) &= 0.
\label{def:zeroV}
\end{align}
Remarking that previous equality can be rewritten as follows,
\begin{align}
\Gamma_m &=  \frac{1}{\EE{\frac{1}{ \nrm{ (X- m)(X-m)^T - \Gamma_m}_F} }}\EE{\frac{ (X-m)(X-m)^T }{ \nrm{ (X- m)(X-m)^T - \Gamma_m}_F}},
\label{def:baseweiszfled}
\end{align}
it is clear that $\Gamma_m$ is a bounded, symmetric and non negative operator in $\mathcal{S}(H)$.

As stated in Proposition~2 of \cite{KrausPanaretos2012}, operator $\Gamma_m$ has an important  stability property when the distribution of $X$ is symmetric, with finite second moment, \textit{i.e} $\EE{\nrm{X}^2}< \infty$. 
Indeed,  the covariance operator of $X$,  $\Sigma = \EE{(X-m)(X-m)^T}$, which is well defined in this case, and $\Gamma_m$ share the same eigenvectors:
%it admits a spectral decomposition (it is trace class)
%\begin{align}
%\Sigma & = \sum_{j \geq 1} \lambda_j e_j e_j^T
%\end{align}
%where the eigenvalues $\lambda_1 \geq \lambda_2 \geq \ldots \geq 0$ are sorted in decreasing order and $e_1, e_2, \ldots$  denote the corresponding orthonormal eigenvectors. 
%If $\lambda_j$ is of multiplicity one, then $e_j$ is also an eigenvector of the median covariation matrix $\Gamma_m$. 
if $e_j$ is an eigenvector of $\Sigma$ with corresponding eigenvalue  $\lambda_j$, then $\Gamma_m e_j = \tilde{\lambda}_j e_j$, for some non negative value $\tilde{\lambda}_j$.
This important result means that for Gaussian and more generally symmetric distribution (with finite second order moments), the covariance operator and the median covariation operator have the same eigenspaces. Note that it is also conjectured in \cite{KrausPanaretos2012} that the order of the eigenfunctions is also the same.

\subsection{Efficient recursive algorithms}

We suppose now that we have i.i.d. copies  $X_1, \ldots, X_n, \ldots$ of random variables with the same law as $X$.
%According to the size of the sample and the dimension of the data under study, we propose two different ways of building estimators of the median covariation matrix. 

For simplicity, we temporarily suppose that the median $m$ of $X$ is known.
We consider a sequence of (learning) weights $\gamma_n = c_\gamma / n^{\alpha}$, with $c_\gamma>0$ and $1/2 <\alpha <1$ and we define the recursive estimation procedure as follows
\begin{align}
W_{n+1} &= W_n + \gamma_n \frac{ (X_{n+1}-m)(X_{n+1}-m)^T - W_n}{ \nrm{(X_{n+1}-m)(X_{n+1}-m)^T - W_n}_F} \label{def:algoRMcov}\\
\overline{W}_{n+1} &= \overline{W}_{n} - \frac{1}{n+1} \left( \overline{W}_{n}  - W_{n+1} \right).
\end{align}
This algorithm can be seen as a particular case of the averaged stochastic gradient algorithm studied in \cite{CCZ11}.
Indeed, the first recursive algorithm (\ref{def:algoRMcov}) is a stochastic gradient algorithm, 
\[
\EE{ \frac{ (X_{n+1}-m)(X_{n+1}-m)^T - W_n}{ \nrm{(X_{n+1}-m)(X_{n+1}-m)^T - W_n}_F} | \CF_n } = \nabla G_m(W_n)
\]
 where $\CF_n =\sigma(X_1, \ldots, X_n)$ is the $\sigma$-algebra generated by $X_1, \ldots, X_n$ whereas the final estimator $\overline{W}_n$ is obtained by averaging the past values of the first algorithm. The averaging step (see \cite{PolyakJud92}), {\textit i.e.}  the computation of  the arithmetical mean of the past values of a slowly convergent estimator (see Proposition~\ref{prop:RMVn} below), permits to obtain a new and efficient estimator converging at a parametric rate, with the same asymptotic variance as the empirical risk minimizer  (see Theorem~\ref{theo:asymptnorm} below).

In most of the cases the value of $m$ is unknown so that it also required to estimate the median. 
To build an estimator of $\Gamma_m$, it is possible to estimate simultaneously
 $m$ and $\Gamma_m$ by considering two averaged stochastic gradient algorithms that are running simultaneously. For $n \geq 1$,
\begin{align}
m_{n+1} & = m_n + \gamma_n^{(m)} \frac{ X_{n+1}-m_n}{\nrm{X_{n+1}-m_n}} \nonumber \\
\overline{m}_{n+1} &= \overline{m}_{n} - \frac{1}{n+1} \left( \overline{m}_{n}  - m_{n+1} \right) \label{def:medaver} \\
V_{n+1} &= V_n + \gamma_n \frac{ (X_{n+1}-\overline{m}_n)(X_{n+1}-\overline{m}_n)^T - V_n}{ \nrm{(X_{n+1}-\overline{m}_n)(X_{n+1}-\overline{m}_n)^T - V_n}_F}  \label{def:Gammarm} \\
\overline{V}_{n+1} &= \overline{V}_{n} - \frac{1}{n+1} \left( \overline{V}_{n}  - V_{n+1} \right), \label{def:Gammamedaver}
\end{align}
where the averaged recursive estimator $\overline{m}_{n+1}$ of the median $m$ is controlled by a sequence of descent steps $ \gamma_n^{(m)}$. The learning rates are generally chosen as follows, $ \gamma_n^{(m)} = c_m n^{-\alpha}$, where the tuning constants satisfy $c_m \in [2,20]$ and $1/2 < \alpha < 1$.

Note  that by  construction, even if $V_n$ is non negative, $V_{n+1}$ may not  be a non negative matrix  when the learning steps do not satisfy
\[
\frac{\gamma_n}{\nrm{(X_{n+1}-\overline{m}_n)(X_{n+1}-\overline{m}_n)^T - V_n}_F} \leq 1 .
\]
 Projecting  $V_{n+1}$ onto the closed convex cone of non negative operators would require to compute the eigenvalues  of $V_{n+1}$ which is time consuming in high dimension even if $V_{n+1}$ is a rank one perturbation to $V_n$ (see \cite{CD2015}).  We consider the following  simple approximation to this projection which consists in replacing   in (\ref{def:Gammarm}) the descent step $\gamma_n$  by a thresholded one,
\begin{align}
\gamma_{n,pos} &= \min \left( \gamma_n , \ \nrm{(X_{n+1}-\overline{m}_n)(X_{n+1}-\overline{m}_n)^T - V_n}_F \right)
\label{def:gammamodif}
\end{align}
which ensures that  $V_{n+1}$ remains non negative when $V_n$ is non negative. The use of these modified steps and an initialization of the recursive algorithm (\ref{def:Gammarm}) with a non negative matrix (for example $V_0=0$) ensure that for all $n \geq 1$, $V_n$ and $\overline{V}_n$ are non negative.

\subsection{Online estimation of the principal components}

It is also possible to approximate recursively the $q$ eigenvectors (unique up to sign) of $\Gamma_m$ associated to the $q$ largest eigenvalues without being obliged to perform a spectral decomposition of $\overline{V}_{n+1}$ at each new observation. Many recursive strategies can be employed (see 
\cite{CD2015} for a review on various recursive estimation procedures of  the eigenelements of a covariance matrix). Because of its simplicity and its accuracy, we consider the following one: 
\begin{align}
 u_{j,n+1} &= u_{j,n} + \frac{1}{n+1} \left( \overline{V}_{n+1} \frac{u_{j,n}}{\| u_{j,n}\|} - u_{j,n} \right), \quad j=1, \ldots, q 
 \label{algo:vectp}
\end{align}
combined with an orthogonalization by deflation of  $u_{1,n+1}, \ldots u_{q,n+1}$.
This recursive algorithm  is based on ideas developed by \cite{Wengetal2003} that are related to the power method for extracting eigenvectors. If we assume that the $q$ first eigenvalues $\lambda_1 > \cdots > \lambda_q$ are distinct, the estimated eigenvectors  $u_{1,n+1}, \ldots u_{q,n+1}$, which are uniquely determined up to sign change, tend to $\lambda_1 u_1, \ldots, \lambda_q u_q.$

Once the eigenvectors are computed, it is possible to compute the principal components as well as indices of outlyingness for each new observation (see \cite{HRVA2008} for a review of outliers detection with multivariate approaches).
 
%Note that an exact alternative there  \cite{GuEisenstat94}

\subsection{Practical issues, complexity and memory}

The recursive algorithms (\ref{def:Gammarm}) and (\ref{def:Gammamedaver}) require each  $O(d^2)$ elementary operations at each update. With the additional online estimation given in (\ref{algo:vectp}) of the $q$ eigenvectors associated to the $q$ largest eigenvalues, $O(qd^2)$ additional operations are required.
The orthogonalization procedure only requires $O(q^2d)$ elementary operations. 

Note that the use of classical Newton-Raphson algorithms for estimating the MCM (see \cite{FFC2012}) can not be envisaged for high dimensional data since the computation or the approximation of the Hessian matrix would require  $O(d^4)$ elementary operations. The well known and fast Weiszfeld's algorithm requires $O(nd^2)$ elementary operations for each sample with size $n$. However, the estimation cannot be updated automatically if the data arrive sequentially. Another drawback compared to the recursive algorithms studied in this paper is that all the data must be stored in memory, which is of order $O(nd^2)$ elements whereas the recursive technique require an amount of memory of order $O(d^2)$. 
%As noted in \cite{CCZ11}, the recursive approaches are much faster as $n$ and $d$ are large.

The  performances of the recursive algorithms depend on the values of tuning parameters $c_\gamma$, $c_m$ and $\alpha$. The value of parameter $\alpha$ is often chosen to be $\alpha=2/3$ or $\alpha=3/4$. Previous empirical studies (see \cite{CCZ11}  and \cite{CardCC10}) have shown that, thanks to the averaging step, estimator $\overline{m}_n$ performs well and is not too sensitive to the choice of $c_m$, provided that the value of $c_m$ is not  too small.  An intuitive explanation could be  that here the recursive process is in some sense "self-normalized" since the deviations at each iteration in (\ref{def:algoRMcov}) have unit norm and  finding  some universal values for $c_m$ is possible.   Usual values for $c_m$ and $c_\gamma$  are in the interval $[2,20]$.
 When $n$ is fixed, this averaged recursive algorithm is about 30 times faster than the Weiszfeld's approach (see \cite{CCZ11}).

\section{Asymptotic properties} 

When $m$ is known, $\overline{W}_n$ can be seen as an averaged stochastic gradient estimator  of the geometric median in a particular Hilbert space and the asymptotic weak convergence of  such estimator has been studied in \cite{CCZ11}. They have  shown that:%For example, the following theorem states the asymptotic normality of the estimator,

\begin{thm} (\cite{CCZ11}, Theorem 3.4). \label{theo:asymptnorm} \\
If assumptions 1-3(a) hold, then as $n$ tends to infinity,
\[
\sqrt{n} \left(\overline{W}_n - \Gamma_m \right)   \rightsquigarrow \mathcal{N}(0, \Delta)
\]
%and
%\[
%\sqrt{n} \left(\widehat{\Gamma}_n - \Gamma_m\right)   \rightsquigarrow \mathcal{N}(0, \Delta)
%\]
where $\rightsquigarrow$ stands for convergence in distribution and $\Delta = \left(\nabla _m^2 (\Gamma_m)\right)^{-1} \Psi  \left(\nabla _m^2 (\Gamma_m)\right)^{-1}$  is the limiting  covariance operator, with $\Psi = \EE{\frac{(Y(m) - \Gamma_m) \otimes_F  (Y(m) - \Gamma_m) }{{\nrm{Y(m) - \Gamma_m}_F}^2}}.$
\end{thm}  
%The proof of the convergence in distribution of $\overline{W}_n$ is given in \cite{CCZ11}. 
As explained in \cite{CCZ11}, the estimator  $\overline{W}_n$ is efficient in the sense that it has the same asymptotic distribution as the empirical risk minimizer related to $G_m(V)$  (see for the derivation of its asymptotic normality in  \cite{MNO2010} in the multivariate case and \cite{ChaCha2014} in a more general functional framework).

Using the  delta method for weak convergence  in Hilbert spaces (see \cite{DauxoisPousseRomain82} or \cite{CGER2007}), one can deduce, from Theorem \ref{theo:asymptnorm}, the asymptotic normality of the estimated eigenvectors of $\overline{W}_n$.
It can also be proven (see \cite{godichon2015}), under Assumptions 1-3, that there is a positive constant $K$ such that for all $n \geq 1$,
\[
\mathbb{E}\left[ \left\| \overline{W}_{n} - \Gamma_{m} \right\|_{F}^{2} \right] \leq \frac{K}{n}.
\]
Note finally that non asymptotic bounds for the deviation of $\overline{W}_n$ around $\Gamma_m$ can be derived readily with the general results given in \cite{CCG2015}.
%Note also that a direct application of ... permits to  prove the convergence in quadratic mean of  $\overline{V}_n$ towards $\Gamma_m$ (see ) 

\medskip

%\subsection{The geometric median $m$ is unknown}
The more realistic case in which $m$ must also be estimated is more complicated because $\overline{V}_n$ depends on $\overline{m}_n$ which is also estimated recursively with the same data. 
 We first  state the strong consistency of the estimators $V_n$ and $\overline{V}_n$.
\begin{thm}\label{theops}
If assumptions 1-3(b) hold, we have
\begin{align*}
\lim_{n \rightarrow \infty}\left\|V_{n} -\Gamma_{m}\right\|_{F}=0 \quad a.s.
\end{align*}
and 
\begin{align*}
\lim_{n \rightarrow \infty}\left\| \overline{V}_{n} - \Gamma_{m}\right\|_{F} &=0 \quad a.s.
\end{align*}
\end{thm}

The obtention of the rate convergence of the averaged recursive algorithm relies on a fine control of the asymptotic behavior of the Robbins-Monro algorithms, as stated in the following proposition.
\begin{thm}\label{theol2l4}
If assumptions 1-3(b) hold, there is a positive constant $C'$, and for all $\beta \in \left( \alpha , 2 \alpha\right)$, there is a positive constant $C_{\beta}$ such that for all $n \geq 1$,
\begin{align*}
 & \mathbb{E}\left[ \left\| V_{n} - \Gamma_{m} \right\|_{F}^{2}\right] \leq \frac{C'}{n^{\alpha}}, \\
 & \mathbb{E}\left[ \left\| V_{n+1} -  \Gamma_{m} \right\|_{F}^{4}\right] \leq \frac{C''}{n^{\beta}}.
\end{align*}
\end{thm}

The obtention of an upper bound for the rate of convergence at the order four of the Robbins-Monro algorithm is crucial in the proofs. Furthermore, the following proposition ensures that the exhibited rate  in quadratic mean is the optimal one.
\begin{prop}
Under assumptions 1-3(b), there is a positive constant $c'$ such that for all $n \geq 1$, 
\begin{align*}
\mathbb{E}\left[ \left\| V_{n} - \Gamma_{m} \right\|_{F}^{2}\right] \geq \frac{c'}{n^{\alpha}}. 
\end{align*}
\label{prop:RMVn}
\end{prop}

%\end{comment}
%%%%%%%%%%%%%%%%%%%%%%%%%%%%
Finally, the following theorem is the most important theoretical result of this work. It shows that, in spite of the fact that it only considers the observed data one by one, the averaged recursive estimation procedure gives an estimator which has a classical parametric $\sqrt{n}$ rate of convergence in the Hilbert-Schmidt norm. 
\begin{thm}\label{th:cvgeqm}
Under Assumptions 1-3(b), there is a positive constant $K'$ such that for all $n \geq 1$,
\begin{align*}
\mathbb{E}\left[ \left\| \overline{V}_{n} - \Gamma_{m} \right\|_{F}^{2}\right] &\leq \frac{K'}{n}.
\end{align*}
\end{thm}

Assuming the eigenvalues of $\Gamma_m$ are of multiplicity one, it can be deduced from Theorem~\ref{th:cvgeqm}  and   Lemma 4.3 in \cite{Bosq}, the convergence in quadratic mean of the eigenvectors of $\overline{V}_{n}$ towards the corresponding (up to sign) eigenvector of $\Gamma_m$ .

%%%%%%%%%%%%%%%%%%%%%%%%%%%%%%%%%%%%%%%%%%%%%
\section{An illustration on simulated and real data}

A small comparison with other classical robust PCA techniques is performed in this section considering data in relatively high dimension but samples with moderate sizes. This permits to compare our approach with classical robust PCA techniques, which are generally not designed to deal with large samples of high dimensional data. In our comparison, we have employed the following well known robust techniques: robust projection pursuit  (see \cite{CR-G2005} and \cite{CFO2007}), minimum covariance determinant  (MCD, see \cite{RvD99}) and spherical PCA (see \cite{LMSTZC1999}). 
The computations were made in the R language  (\cite{R10}), with the help of packages \texttt{pcaPP} and \texttt{rrcov}. For reproductible research, our codes  for computing the MCM have been posted on CRAN in the \texttt{Gmedian} package. We will denote by MCM(R) the recursive estimator $\overline{V}_{n}$ defined in (\ref{def:Gammamedaver}) and MCM(R+) its non negative modification whose learning weights are defined in (\ref{def:gammamodif}).

If the size of the data $n \times d$ is not too large, an effective way for estimating $\Gamma_m$ is to employ  Weiszfeld's algorithm (see \cite{Weiszfeld1937} and \cite{VZ00} as well the  Supplementary file  for a description of the algorithms in our particular situation). The estimate obtained thanks to Weiszfeld's algorithm is denoted by MCM(W) in the following. Note that other optimization algorithms which may be preferred in small dimension  (see  \cite{FFC2012}) have not been considered here since they would require the computation of the Hessian matrix, whose size is $d^4$, and this would lead to much slower algorithms. Note finally that all these alternative algorithms do not admit a natural updating scheme when the data arrive sequentially so that they should be completely ran again at each new observation.
 
%Note that according to \cite{FFC2012} the NLM algorithm seems to perform better in this situation. However since it requires the computation (or an approximation) of the Hessian matrix, it is not adapted to deal with high dimensional data 

\subsection{Simulation protocol}

Independent realizations of  a random variable $Y \in \mathbb{R}^d$ are drawn, where 
\begin{eqnarray}
Y &=& (1-O(\delta)) X + O(\delta) \epsilon,
\label{def:melange}
\end{eqnarray}
is a mixture of two distributions and $X, O$ and $\epsilon$ are independent random variables.  The random vector $X$ has a centered Gaussian distribution in $\mathbb{R}^d$ with covariance matrix $[\Sigma]_{\ell, j} = \min (\ell,j)/d$ and can be thought as a discretized version of a Brownian sample path in $[0,1]$. The multivariate contamination comes from $\epsilon$, with different rates of contamination controlled by the Bernoulli variable $O(\delta)$, independent from $X$ and $\epsilon$, with $\mathbb{P}(O(\delta) =1) =  \delta$ and $\mathbb{P}(O(\delta) =0) =  1-\delta$. Three different scenarios (see Figure~\ref{fig:traj}) are considered for the distribution of $\epsilon$:
\begin{itemize}
\item The elements of vector $\epsilon$ are $d$ independent realizations of a  Student $t$ distribution with one degree of freedom. This means that the first moment of $Y$ is not defined when $\delta>0$.
\item The elements of vector $\epsilon$ are $d$ independent realizations of a  Student $t$ distribution with two degrees of freedom. This means that the second  moment of $Y$ is not defined when $\delta>0$.
\item The vector $\epsilon$ is distributed as a "reverse time" Brownian motion. It has a Gaussian centered distribution, with covariance matrix  $[\Sigma_\epsilon]_{\ell, j} = 2\min (d-\ell,d-j)/d$. The covariance matrix of $Y$ is $(1-\delta) \Sigma + \delta \Sigma_\epsilon$.
\end{itemize}

\begin{figure}
\begin{center}
\hspace*{-2cm}
\includegraphics[height=9cm,width=18cm]{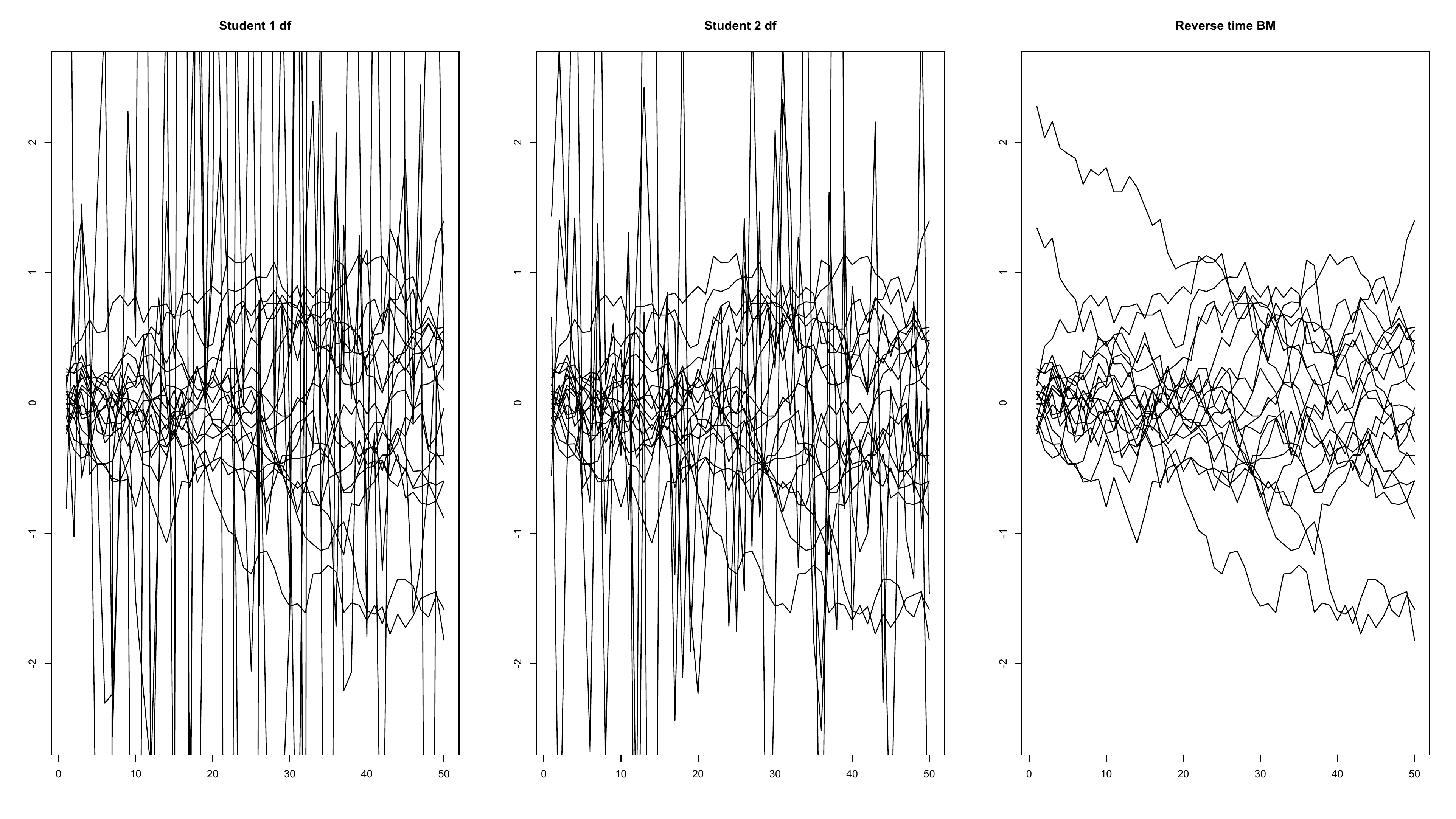}%
 \caption{A sample of $n=20$ trajectories when $d=50$ and $\delta=0.10$ for the three different contamination scenarios: Student $t$ with 1 degree of freedom, Student $t$ with 2 degrees of freedom and reverse time Brownian motion (from left to right).}
 \label{fig:traj}
 \end{center}
\end{figure}

% For the Weiszfeld's algorithm (\ref{def:algoiter}), the median $m$ is replaced by an estimation. 
 For the averaged recursive algorithms, we have considered  tuning coefficients $c_m=c_\gamma = 2$ and a speed rate of $\alpha=3/4$. Note that the values of these tuning parameters have not been particularly optimised. We have noted that the simulation results were very stable, and did not depend much on the value of $c_m$ and $c_\gamma$ for $c_m, c_\gamma \in [1,20]$.

The  estimation error of the eigenspaces associated to the largest eigenvalues is evaluated by considering the squared Frobenius norm between the associated orthogonal projectors. 
Denoting by $\mathbf{P}_q$ the orthogonal projector onto the space generated by the $q$ eigenvectors of the covariance matrix $\Sigma$ associated to the $q$ largest eigenvalues and by $\widehat{\mathbf{P}}_q$ an estimation, we consider the following loss criterion, 
\begin{align}
R(\widehat{\mathbf{P}}_q, \mathbf{P}_q) &= \mbox{tr} \left[ \left( \widehat{\mathbf{P}}_q - \mathbf{P}_q \right)^T \left( \widehat{\mathbf{P}}_q - \mathbf{P}_q \right) \right] \nonumber \\
 &= 2 q - 2 \mbox{tr} \left[ \widehat{\mathbf{P}}_q \mathbf{P}_q \right].
\label{def:errvecp} 
\end{align}
Note that we always have  $R(\widehat{\mathbf{P}}_q, \mathbf{P}_q) \leq 2q$ and $R(\widehat{\mathbf{P}}_q, \mathbf{P}_q)=2q$ means that the  eigenspaces generated by the true and the estimated eigenvectors are orthogonal.

\subsection{Comparison with classical robust PCA techniques}

\begin{figure}
\includegraphics[height=10cm,width=15.5cm]{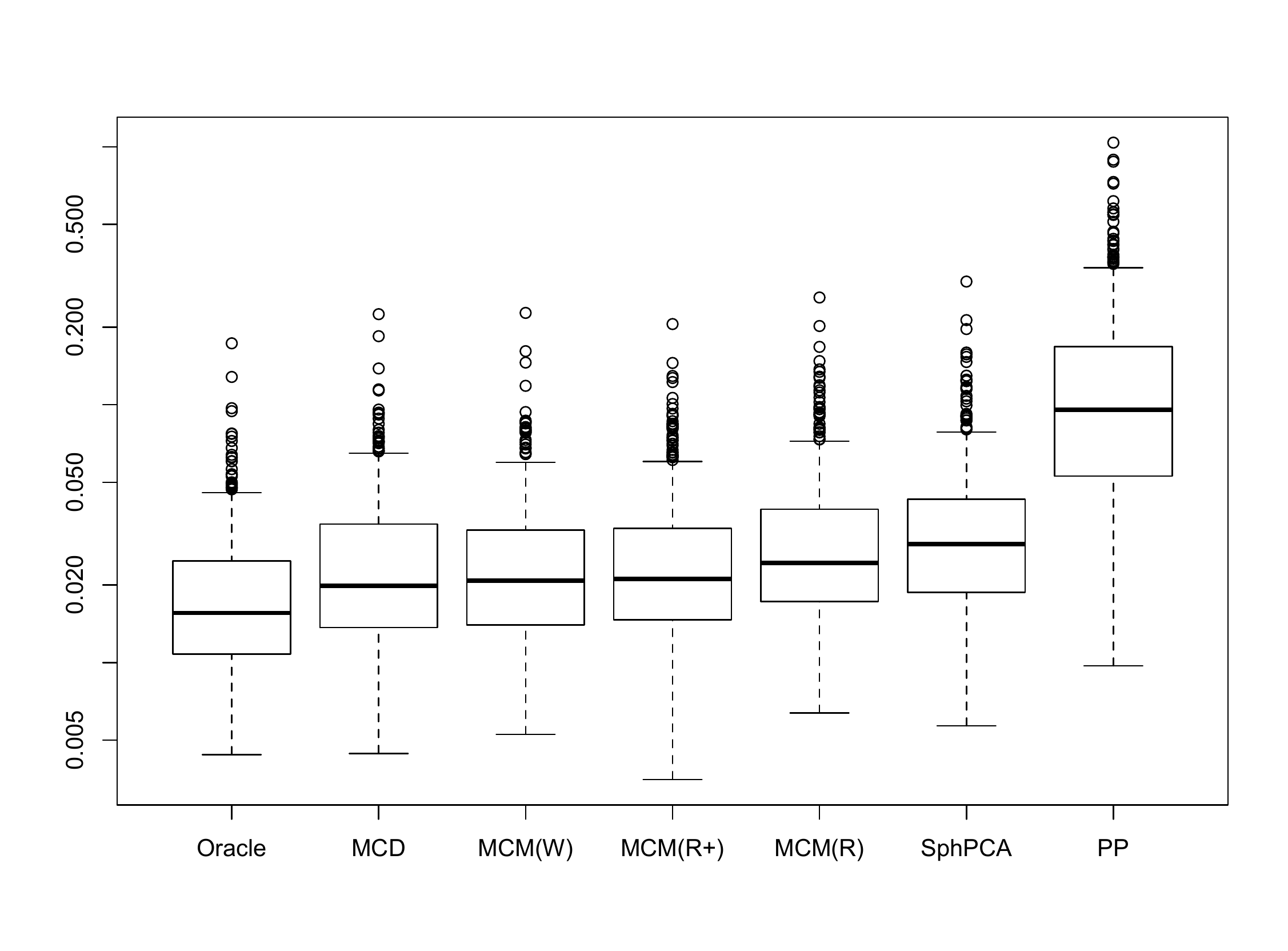}
 \caption{Estimation errors (at a logarithmic scale) over 500 Monte Carlo replications, for $n=200$, $d=50$  with no contamination ($\delta=0$). MCM(W) stands for the estimation performed by the Weiszfeld's algorithm whereas MCM(R) denotes the averaged recursive approach and MCM(R+) its non negative modification (see equation \ref{def:gammamodif}).}
 \label{fig:boxerr0}
  \end{figure}

\begin{figure}
\includegraphics[height=10cm,width=15.5cm]{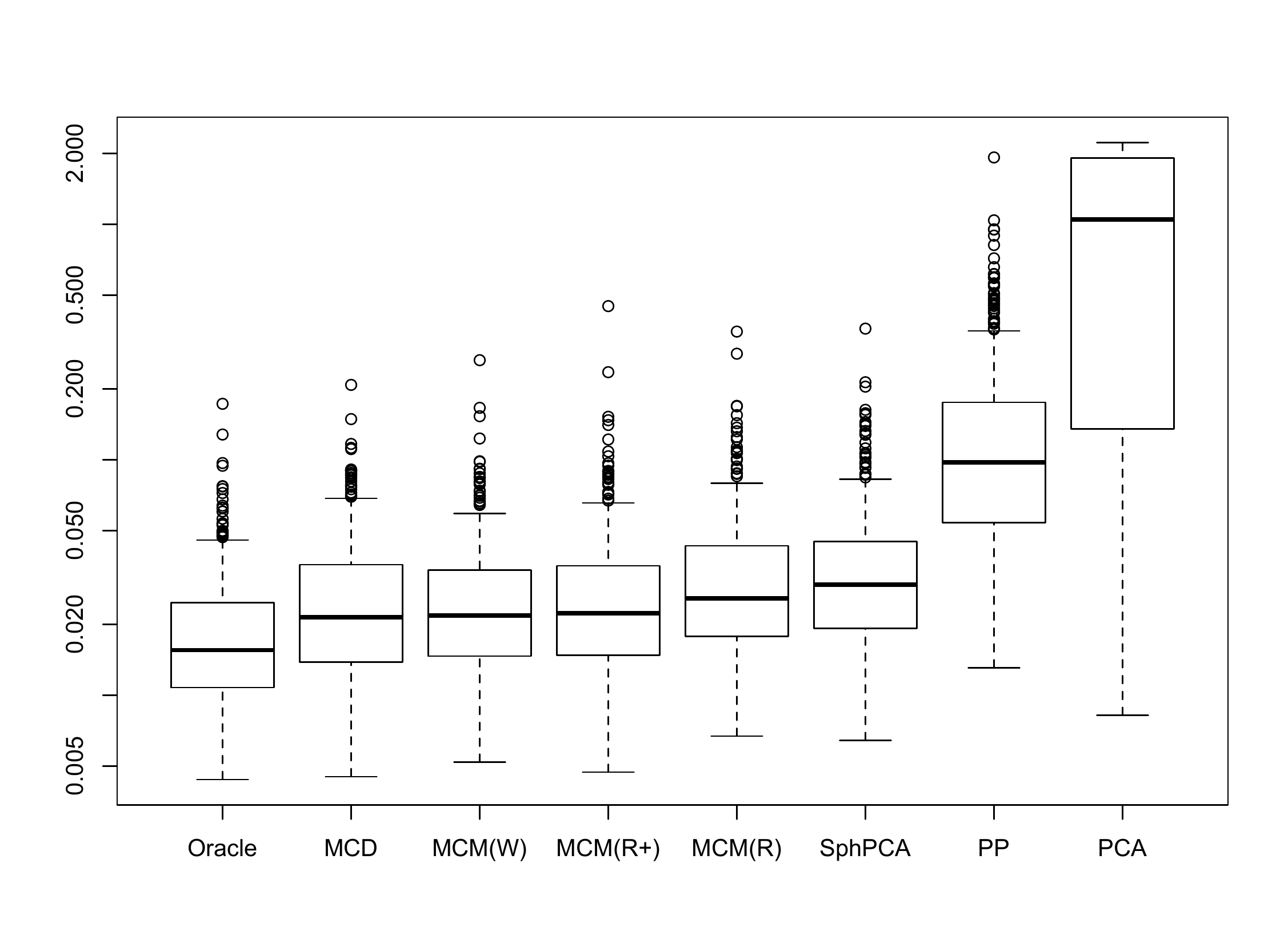}
 \caption{Estimation errors (at a logarithmic scale) over 500 Monte Carlo replications, for $n=200$, $d=50$ and a contamination by a $t$ distribution with 2 degrees of freedom  with $\delta=0.02$. MCM(W) stands for the estimation performed by the Weiszfeld's algorithm whereas MCM(R) denotes the averaged recursive approach and MCM(R+), its non negative modification with learning steps as in (\ref{def:gammamodif}).}
 \label{fig:boxerr}
  \end{figure}

We first compare the performances of  the two estimators of the  MCM based on the Weiszfeld's algorithm and the recursive algorithms  (see (\ref{def:Gammamedaver})) with more classical robust PCA techniques.

We generated  samples of $Y$ with size $n=500$  and dimension $d \in \{50,200\}$, over 500 replications. Different levels of contamination are considered : $\delta \in \{0, 0.02, 0.05, 0.10, 0.20 \}$. For both dimensions $d =50$ and $d=200$, the first eigenvalue of the covariance matrix of $X$ represents about 81 \% of the total variance, and the second one about 9 \%. %Our conclusions do not differ much for different sample sizes.

 \begin{table}[t]
   \centering
   \small
   \begin{tabular}{cccccccc} \\
 & PCA & MCD & MCM(W) & MCM(R+) & MCM(R) &   SphPCA & PP \\
d=50 & 0.0156 & 0.0199 & 0.0208 & 0.0211& 0.0243 & 0.0287 & 0.0955\\
 d=200 &  0.0148 & - & 0.0200& 0.0209& 0.0246& 0.0275& 0.0895 \\ 
       \end{tabular}
   \caption{Median estimation errors, according to criterion $R(\widehat{\mathbf{P}}_q, \mathbf{P}_q) $ with a dimension $q=2$, for non contaminated  samples of  size $n=200$, over 500 Monte Carlo experiments.}
   \label{tab:summary_mu0}
  \end{table}

The median errors of estimation of the eigenspace generated by the first two eigenvectors ($q=2$), according to criterion (\ref{def:errvecp}), are given in Table~\ref{tab:summary_mu0} for non contaminated data ($\delta=0$). The distribution of the  estimation error $R(\widehat{\mathbf{P}}_q, \mathbf{P}_q)$ is drawn for the different approaches in Figure~\ref{fig:boxerr0} when the dimension is not large ($d=50$). As expected, the "Oracle", which is the classical PCA in this situation, provides the best estimations of the eigenspaces. Then, the MCD and the median covariation matrix, estimated by the Weiszfeld algorithm or the modified MCM(R+) recursive estimator, behave well and similarly. Note that when the dimension gets larger, the MCD cannot be used anymore and the MCM is the more effective robust estimator of the eigenspaces.

When the data are contaminated, the median errors of estimation of the eigenspace generated by the first two eigenvectors ($q=2$), according to criterion (\ref{def:errvecp}), are given in Table~\ref{tab:summary_mu}. In Figure~\ref{fig:boxerr}, the distribution of the  estimation error $R(\widehat{\mathbf{P}}_q, \mathbf{P}_q)$ is drawn for the different approaches.

 \begin{table}
   \centering
   \small
   \begin{tabular}{clcccccc} \\
  \multicolumn{2}{c}{}   & $t$ 1 df & $t$ 2 df & inv. B. &  $t$ 1 df & $t$ 2 df & inv. B.  \\
   % \hline
 $\delta$ & Method  & \multicolumn{3}{c}{d  = 50} & \multicolumn{3}{c}{d  = 200} \\
    2\%  &  PCA &  3.13&  1.04 &  0.698  &   3.95  &  1.87  &   0.731  \\ 
              &  PP &  0.086 &  0.097 &  0.090   &   0.085   &  0.094  &   0.084    \\
              &MCD &  0.022 &  0.021 &  0.021   &  --   & --   &   --   \\
       &Sph. PCA &  0.028 &  0.029 &   0.027  &   0.027   &  0.030  &   0.028    \\
&MCM (Weiszfeld)&  0.021 &  0.021 &  0.021   &   0.021   &  0.022   &   0.022    \\
              &MCM (R+) &  0.022 &  0.022 &  0.024   &   0.023   &  0.023   &   0.025    \\
              &MCM (R) &  0.026 &  0.025 &  0.027   &   0.026   &  0.027   &   0.028    \\
    5\%  &  PCA &  3.82 &  1.91 &  0.862  &   3.96  &  1.98  &   0.910  \\ 
              &  PP &  0.090 &  0.103 &  0.093   &   0.089  &  0.098   &   0.087   \\
              &MCD &  0.022 &  0.023 &  0.021   &   --   &  --   &   --    \\
         &Sph. PCA &  0.029 &  0.031 &  0.033   &   0.029   & 0.031    &   0.034    \\
      &MCM (Weiszfeld)&  0.023 &  0.023 &  0.028   &   0.022   &  0.023  &   0.030    \\
              &MCM (R+) &  0.025 &  0.024 &  0.035   &   0.024   &  0.024   &   0.039    \\
              &MCM (R) &  0.029 &  0.027 &  0.037   &   0.028   &  0.028   &   0.040    \\
    10\%  &  PCA &  3.83 &  1.96 &  1.03  &   3.96  &   1.99  &   1.10 \\ 
              &  PP &  0.107 &  0.108 &  0.099  &   0.088  &  0.101   &   0.097    \\
              &MCD & 0.023 &  0.022 &  0.023   &   --   &  --  &   --   \\
              &Sph. PCA &  0.033 &  0.033 &   0.054  &   0.031   &  0.033   &   0.057   \\
    &MCM (Weiszfeld)&  0.025 &  0.026 &  0.059   &   0.023   &  0.024   &   0.056    \\
              &MCM (R+) &  0.030 &  0.027 & 0.089   &   0.027   &  0.027   &   0.086    \\
              &MCM (R) &  0.035 &  0.032 & 0.088   &   0.032   &  0.031   &   0.086    \\
    20\%  &  PCA &  3.84 &  2.02 &  1.19  &   3.96   &  2.01   &   1.25  \\ 
              &  PP &  0.110 &  0.135 &  0.138   &   0.091   &  0.122   &   0.137    \\
              &MCD &  0.025 &  0.026 &  0.026   &  --   & --  &   --    \\
              &Sph. PCA &  0.037 &  0.038 &  0.140   &   0.034  &  0.037   & 0.150      \\
     &MCM (Weiszfeld)&  0.030 &  0.030 &  0.174   &   0.026  &  0.028  &   0.181    \\
              &MCM (R+) &  0.044  &  0.036 &  0.255   &   0.038   &  0.032   &   0.256 \\ 
               &MCM (R) &  0.050  &  0.041 &  0.251   &   0.042   &  0.037   &   0.256 \\ 
       \end{tabular}
   \caption{Median estimation errors, according to criterion $R(\widehat{\mathbf{P}}_q, \mathbf{P}_q) $ with a dimension $q=2$, for datasets with a sample size $n=200$, over 500 Monte Carlo experiments.}
   \label{tab:summary_mu}
  \end{table}

We can make the following remarks. At first note that even when the level of contamination is small (2\% and 5\%), the performances of classical PCA are strongly affected by the presence of outlying values in such (large) dimensions.  When $d=50$, the MCD algorithm and the MCM estimation provide the best estimations of the original two dimensional eigenspace, whereas when $d$ gets larger ($d=n=200$), the MCD estimator can not be used anymore (by construction)  and the MCM estimators, obtained with Weiszfeld's and the non negative recursive algorithm, remain the most accurate. We can also remark that the recursive MCM algorithms, which are designed to deal with very large samples, performs well even for such moderate sample sizes (see also Figure~\ref{fig:boxerr}).  The modification of the descent step suggested in (\ref{def:gammamodif}), which corresponds to estimator MCM(R+), permits to improve the accuracy the initial MCM estimator, specially when the noise level is not small.  
The performances of the spherical PCA are slightly less accurate whereas the median error of the robust PP is always the largest among the robust estimators. 
When, the contamination is highly structured temporally and the level of contamination is not small (contamination by a reverse time Brownian  motion, with $\delta=0.20$), the behavior of the MCM is different from the other robust estimators and, with our criterion, it can appear as less effective. However, one can think that we are in presence of two different populations with completely different multivariate correlation structure and the MCD completely ignores that part of the data, which is not necessarily a better behavior. 

%As the level of contamination increases (10\%  and  25\%), the performances of MCM deteriorates and the method is outperformed by MCD. Nevertheless, it still behaves better, according to our criterion, than the projection pursuit approach.
%Note finally that in the particular case in which the level of contamination by the reverse time Brownian motion is not small ($\delta=0.20$) the difference between MCD and the other approaches is important. On the one hand, the MCD approach consider the data generated by the reverse time brownian motion as outlying data and eliminate them in the estimation procedure of the covariance. On the other hand,  

\subsection{Online estimation of the principal components}

We now consider an experiment in high dimension, $d=1000$, and evaluate the ability of the recursive algorithms defined in (\ref{algo:vectp}) to estimate recursively the eigenvectors of $\Gamma_m$ associated to the largest eigenvalues. Note that due to the high dimension of the data and limited computation time, we only make comparison of the recursive robust techniques with the classical PCA. 
For this we generate growing samples  and compute, for each sample size  the approximation error of the different (fast) strategies to the true eigenspace generated by the $q$ eigenvectors associated to the $q$ largest eigenvalues of $\Gamma_m$.

We have drawn in Figure~\ref{fig:evolR}, the evolution of the mean (over 100 replications) approximation error $R(\mathbf{P}_q,\hat{\mathbf{P}}_q)$, for a dimension $q=3$, as a function of the sample size for samples contaminated by a 2 degrees of freedom Student $t$ distribution with a rate $\delta=0.1$. 
 An important fact is that the recursive algorithm which approximates recursively  the eigenelements behaves very well and we can see nearly no difference between the spectral decomposition of $\overline{V}_n$ (denoted by MCM in Figure \ref{fig:evolR}) and the estimates produced with the sequential algorithm (\ref{algo:vectp}) for sample sizes larger than a few hundreds. We can also note that the error made by the classical PCA is always very high and does not decrease with the sample size.

 \begin{figure}%[ht]
   \begin{center}
 \includegraphics[width=13cm]{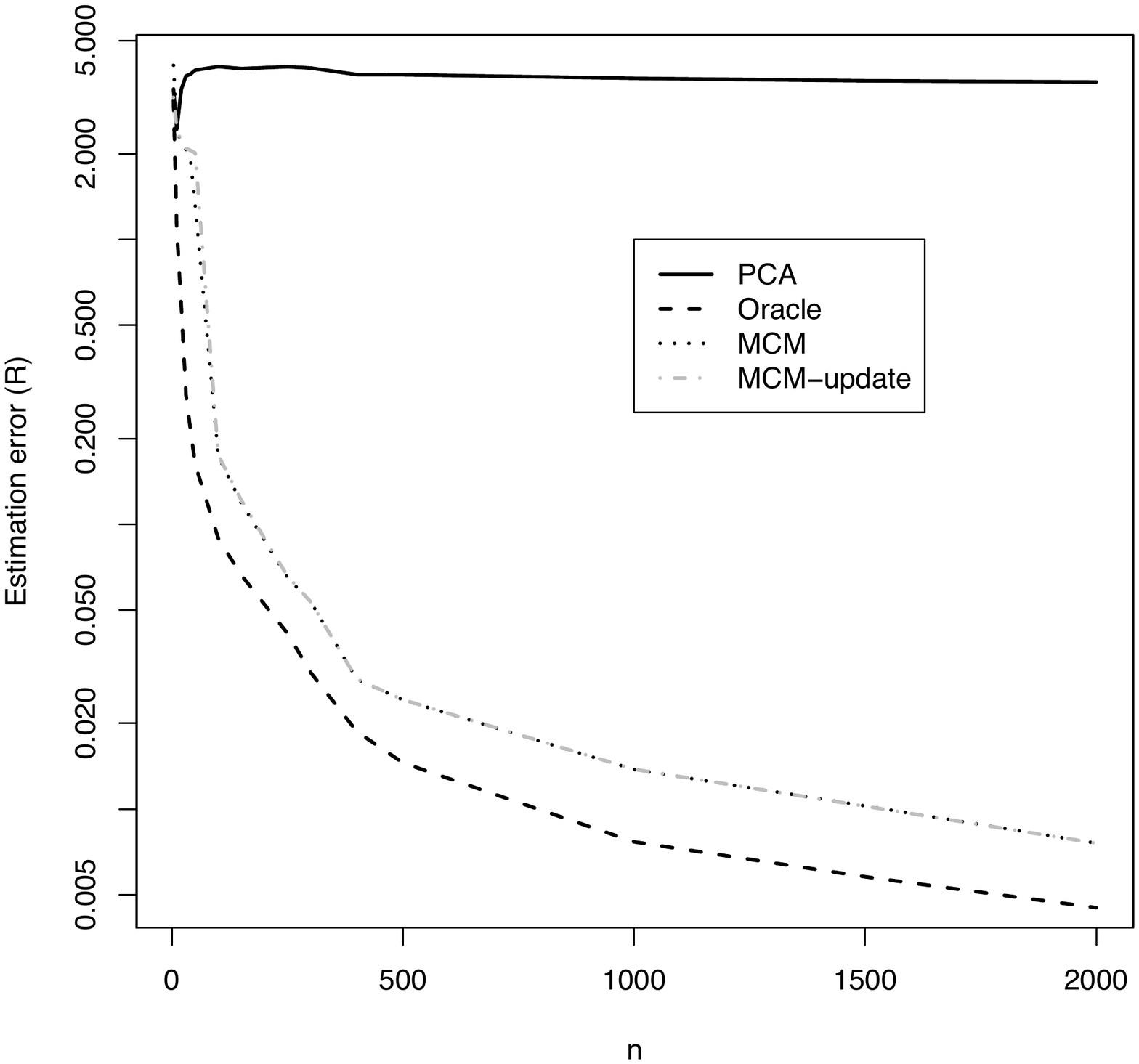}%
 \caption{Estimation errors of the eigenspaces (criterion $R(\widehat{\mathbf{P}}_q$)) with $d=1000$ and $q=3$ for classical PCA, the oracle PCA and the recursive MCM estimator with recursive estimation of the eigenelements (MCM-update) and with static estimation (based on the spectral decomposition of $\overline{V}_{n}$) of the eigenelements (MCM).}
 \label{fig:evolR}
 \end{center}
   \end{figure}

\subsection{Robust PCA of TV audience}
The last example is a high dimension and  large sample case. Individual TV audiences are measured, by the French  company M\'ediam\'etrie, every minutes for a panel of $n=5422$ people over a period of 24 hours, $d=1440$ (see \cite{CCM10} for a more detailed presentation of the data).
With a classical PCA, the first eigenspace represents  24.4\% of the total variability, whereas the second one reproduces  13.5\% of the total variance, the third one 9.64\% and the fourth one 6.79\%. Thus, more than 54\% of the variability of the data can be captured in a four dimensional space. Taking account of the large dimension of the data, these values indicate a high temporal  correlation.

Because of the large dimension of the data, the Weiszfeld's algorithm as well as the other robust PCA techniques can not be used anymore in a reasonable time with a personal computer. The MCM has been computed thanks to the recursive algorithm given in (\ref{def:Gammamedaver}) in approximately 3 minutes on a laptop  in the R language (without any specific C routine).

 \begin{figure}%[ht]
   \begin{center}
 \includegraphics[height=10cm,width=15.5cm]{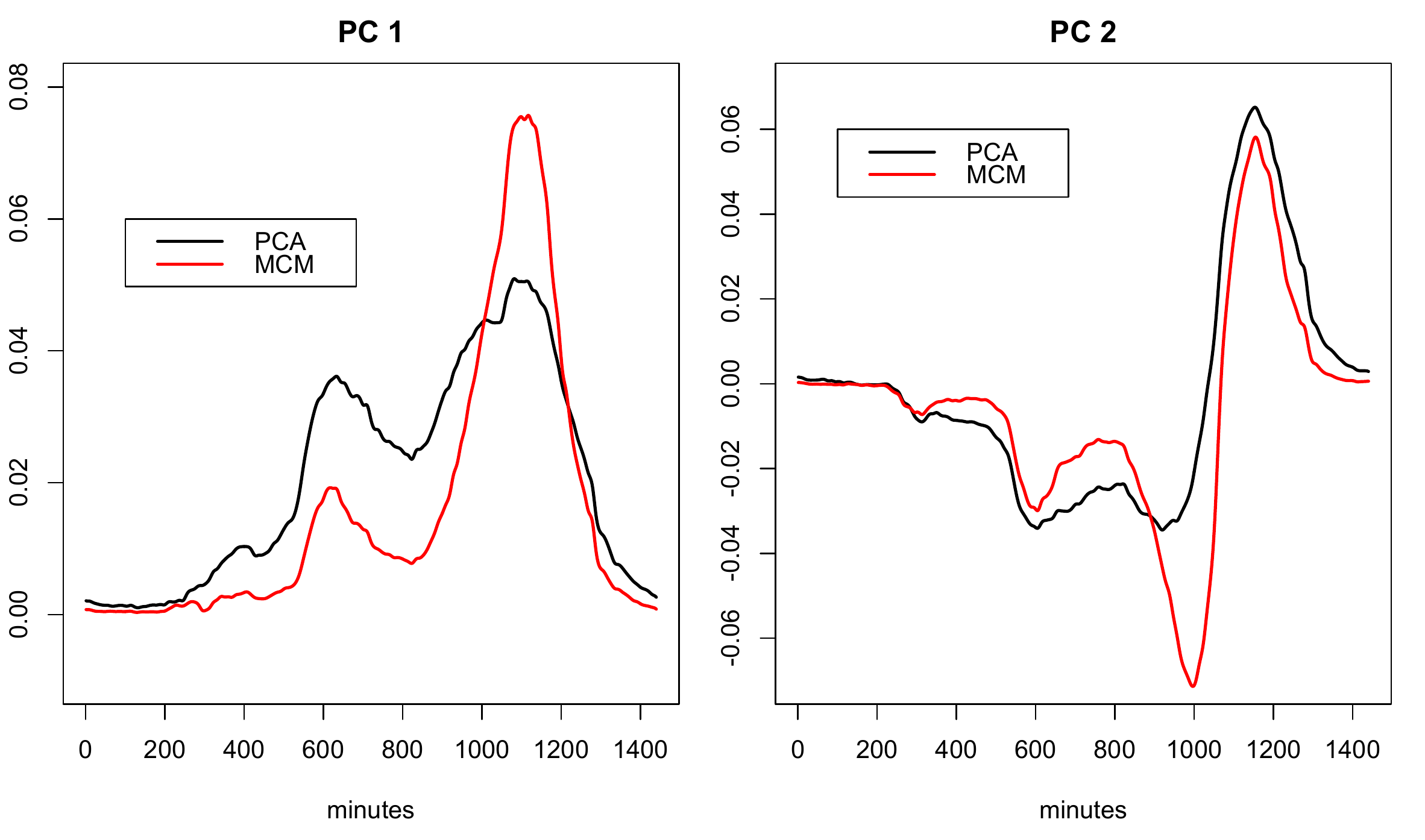}%
 \caption{TV audience data measured the 6th September 2010, at the minute scale. Comparison of the principal components of the classical PCA (black) and robust PCA based on the Median Covariation Matrix (red). First eigenvectors on the left, second eigenvectors on the right.}
 \label{fig1}
 \end{center}
   \end{figure}
   
As seen in Figure~\ref{fig1}, the first two eigenvectors obtained by a classical PCA and the robust PCA based on the MCM are rather different. 
This is confirmed by the relatively large distance between the two corresponding eigenspaces, $R(\widehat{P}_2^{PCA},  \widehat{P}_2^{MCM}) = 0.56$. The first robust eigenvector puts the stress on the time period comprised between 1000 minutes and 1200 minutes whereas the first non robust eigenvector focuses, with a smaller intensity, on a larger period of time comprised between 600 and 1200 minutes. The second robust eigenvector differentiates  between people watching TV during the period between 890 and 1050 minutes (negative value of the second principal component)  and people watching TV between minutes 1090 and 1220 (positive value of the second principal component). Rather surprisingly, the third and fourth eigenvectors of the non robust and robust covariance matrices look quite similar (see Figure~\ref{fig2}).

 \begin{figure}%[ht]
   \begin{center}
  \includegraphics[height=10cm,width=15.5cm]{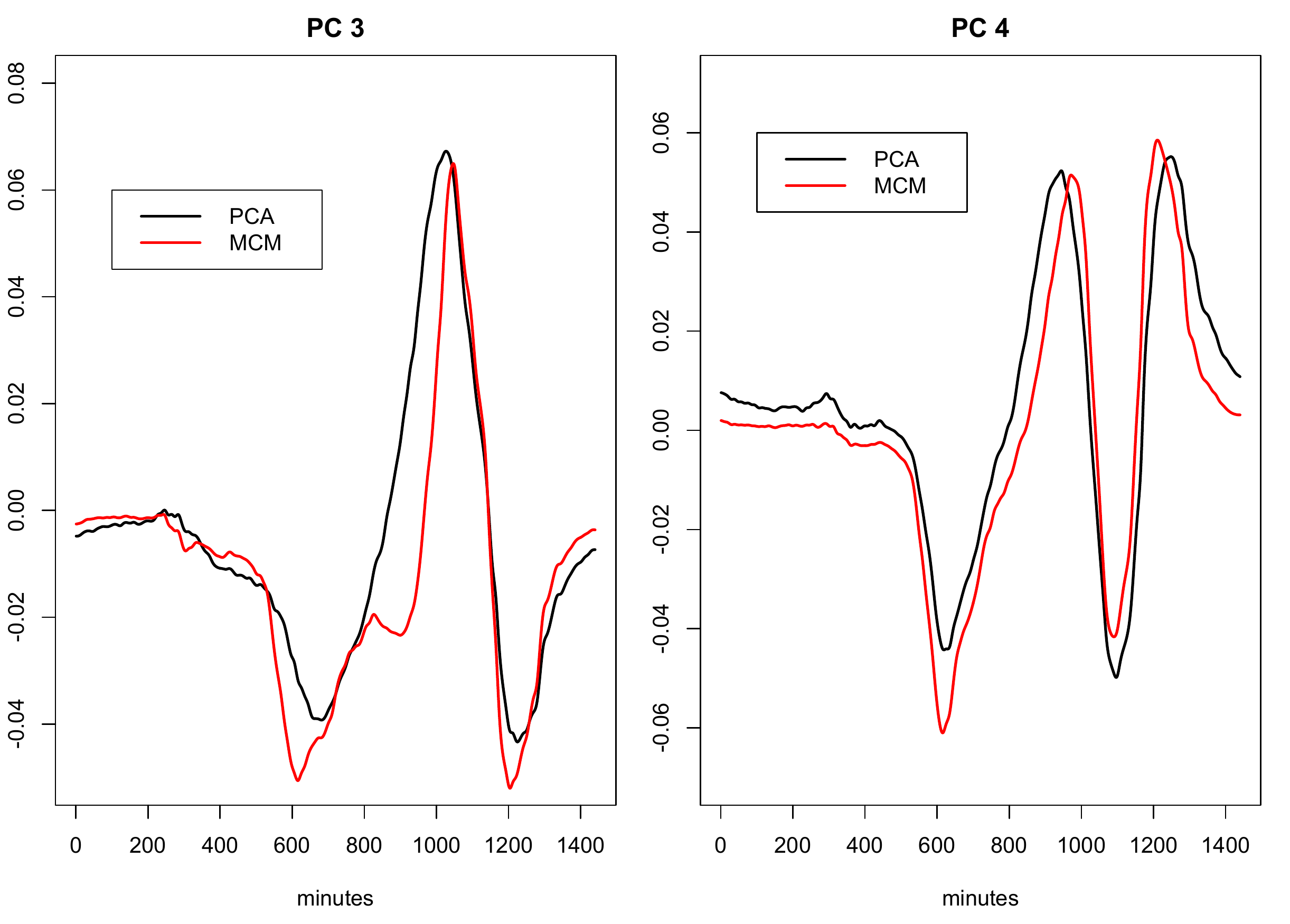}%
 \caption{TV audience data measured the 6th September 2010, at the minute scale. Comparison of the principal components of the classical PCA (black) and robust PCA based on the MCM (red). Third eigenvectors on the left, fourth eigenvectors on the right.}
 \label{fig2}
 \end{center}
   \end{figure}

\section{Proofs}
We give in this Section the proofs of Theorems \ref{theops}, \ref{theol2l4} and  \ref{th:cvgeqm}. These proofs rely on several technical Lemmas whose proofs are given in the Supplementary file.

\subsection{Proof of Theorem \ref{theops}}

Let us recall the Robbins-Monro algorithm, defined recursively by
\begin{align*}
V_{n+1} & =  V_{n} + \gamma_{n} \frac{\left( X_{n+1} - \overline{m}_{n} \right)\left( X_{n+1} - \overline{m}_{n} \right)^{T}-V_{n}}{\left\| \left( X_{n+1} - \overline{m}_{n} \right)\left( X_{n+1} - \overline{m}_{n} \right)^{T}-V_{n} \right\|_{F}} \\
& = V_{n} - \gamma_{n} U_{n+1},
\end{align*}
with $U_{n+1}:= - \frac{\left( X_{n+1} - \overline{m}_{n} \right)\left( X_{n+1} - \overline{m}_{n} \right)^{T}-V_{n}}{\left\| \left( X_{n+1} - \overline{m}_{n} \right)\left( X_{n+1} - \overline{m}_{n} \right)^{T}-V_{n} \right\|_{F}}$. Since $\mathcal{F}_{n} := \sigma \left( X_{1},...,X_{n} \right)$, we have $\mathbb{E}\left[ U_{n+1}|\mathcal{F}_{n} \right] = \nabla G_{\overline{m}_{n}}(V_{n})$. Thus  $\xi_{n+1}:= \nabla_{\overline{m}_{n}}G(V_{n}) - U_{n+1}$, $\left( \xi_{n} \right)$ is a sequence of martingale differences adapted to the filtration $\left( \mathcal{F}_{n} \right)$. Indeed, $\mathbb{E}\left[ \xi_{n+1} | \mathcal{F}_{n} \right] = \nabla G_{\overline{m}_{n}}(V_{n}) - \mathbb{E}\left[ U_{n+1}|\mathcal{F}_{n} \right] = 0$. The algorithm can be written as follows
\[
V_{n+1} = V_{n} - \gamma_{n} \nabla G_{\overline{m}_{n}}(V_{n}) + \gamma_{n}\xi_{n+1}.
\]
Moreover, it can be considered as a stochastic gradient algorithm because it can be decomposed as follows:
\begin{equation}
\label{decxi} V_{n+1} = V_{n} -  \gamma_{n}\left( \nabla G_{\overline{m}_{n}}(V_{n})- \nabla G_{\overline{m}_{n}} ( \Gamma_{m} )\right) + \gamma_{n}\xi_{n+1} - \gamma_{n}r_{n},
\end{equation}
with $r_{n} := \nabla G_{\overline{m}_{n}} ( \Gamma_{m}) - \nabla G_{m} ( \Gamma_{m})$. Finally, linearizing the gradient, 
\begin{equation}
\label{decdelta} V_{n+1} - \Gamma_{m} = \left( I_{\mathcal{S}(H)} - \gamma_{n} \nabla_{m}^{2}G(\Gamma_{m}) \right) \left( V_{n} -  \Gamma_{m}\right) + \gamma_{n}\xi_{n+1} - \gamma_{n}r_{n} - \gamma_{n}r_{n}' - \gamma_{n}\delta_{n},
\end{equation}
with
\begin{align*}
 r_{n}' & := \left( \nabla_{\overline{m}_{n}}^{2}G\left( \Gamma_{m}\right) - \nabla_{m}^{2}G\left( \Gamma_{m}\right)\right)\left( V_{n}-\Gamma_{m}\right) , \\
   \delta_{n} & := \nabla G_{\overline{m}_{n}}\left( V_{n} \right) - \nabla G_{\overline{m}_{n}}\left( \Gamma_{m} \right) - \nabla_{\overline{m}_{n}}^{2}G\left( \Gamma_{m}\right) \left( V_{n} - \Gamma_{m}\right) .  
\end{align*}

The following lemma gives upper bounds of these remainder terms. Its proof is given in the Supplementary file.
\begin{lem}\label{lem3maj}
Under assumptions 1-3(b), we can bound the three remainder terms. First,

\begin{equation}
\label{majdelta} \left\| \delta_{n} \right\|_{F} \leq 6C \left\| V_{n} - \Gamma_{m} \right\|_{F}^{2}.
\end{equation}
In the same way, for all $n \geq 1$,
\begin{equation}\label{majrn}
\left\| r_{n} \right\|_{F} \leq 4 \left( \sqrt{C} + C\sqrt{\left\|\Gamma_{m}\right\|_{F}}\right) \left\| \overline{m}_{n} - m \right\| .
\end{equation}
Finally, for all $n \geq 1$,
\begin{equation}\label{marn'}
\left\| r_{n}' \right\|_{F} \leq 12 \left( C \sqrt{\left\| \Gamma_{m} \right\|_{F}} + C^{3/4}\right) \left\| \overline{m}_{n} - m \right\| \left\| V_{n} - \Gamma_{m}\right\|_{F}. 
\end{equation}
\end{lem}

%\begin{proof}[Proof of Theorem \ref{theops}]
We deduce from decomposition (\ref{decxi}) that for all $n \geq 1$,
\begin{align*}
\left\| V_{n+1} - \Gamma_{m}\right\|_{F}^{2} & = \left\| V_{n} - \Gamma_{m} \right\|_{F}^{2} -2 \gamma_{n} \left\langle V_{n} - \Gamma_{m} , \nabla G_{\overline{m}_{n}}(V_{n}) - \nabla G_{\overline{m}_{n}}(\Gamma_{m}) \right\rangle_{F} \\
& +\gamma_{n}^{2}\left\| \nabla G_{\overline{m}_{n}}(V_{n}) - \nabla G_{\overline{m}_{n}}(\Gamma_{m}) \right\|_{F}^{2} \\
& + \gamma_{n}^{2}\left\| \xi_{n+1}\right\|_{F}^{2} + 2 \gamma_{n} \left\langle V_{n} - \Gamma_{m} - \gamma_{n} \left( \nabla G_{\overline{m}_{n}}(V_{n}) - \nabla G_{\overline{m}_{n}}(\Gamma_{m}) \right) , \xi_{n+1} \right\rangle_{F} \\
& + \gamma_{n}^{2}\left\| r_{n} \right\|_{F}^{2} -2 \gamma_{n}\left\langle r_{n} , V_{n} - \Gamma_{m} \right\rangle_{F} -2\gamma_{n}^{2} \left\langle r_{n} , \xi_{n+1} - \nabla G_{\overline{m}_{n}}(V_{n}) + \nabla G_{\overline{m}_{n}}(\Gamma_{m}) \right\rangle_{F} . 
\end{align*}
Note that for all $h \in H$ and $V \in \mathcal{S}(H)$ we have $\left\| \nabla G_{h}(V) \right\|_{F} \leq 1$. Furthermore, $\left\| r_{n} \right\|_{F} \leq 2$ and $\left\| \xi_{n+1} \right\|_{F} \leq 2$. Using the fact that $\left( \xi_{n} \right)$ is a sequence of martingale differences adapted to the filtration $\left( \mathcal{F}_{n} \right)$,
\begin{align*}
\mathbb{E}\left[ \left\| V_{n+1} - \Gamma_{m} \right\|_{F}^{2} |\mathcal{F}_{n} \right] & \leq \left\| V_{n} - \Gamma_{m}\right\|_{F}^{2} -2\gamma_{n} \left\langle V_{n} - \Gamma_{m} , \nabla_{\overline{m}_{n}}G \left( V_{n} \right) - \nabla_{\overline{m}_{n}}G\left( \Gamma_{m} \right) \right\rangle_{F} \\
& +  28\gamma_{n}^{2}  -2\gamma_{n} \left\langle r_{n} , V_{n} - \Gamma_{m} \right\rangle_{F}.
\end{align*}
Let $\alpha_{n} = n^{-\beta}$, with $\beta \in ( 1-\alpha ,  \alpha )$, we have
\begin{align}
\label{majvitps}\mathbb{E}\left[ \left\| V_{n+1} - \Gamma_{m} \right\|_{F}^{2} |\mathcal{F}_{n} \right] & \leq \left( 1+\gamma_{n}\alpha_{n} \right) \left\| V_{n} - \Gamma_{m}\right\|_{F}^{2} -2\gamma_{n} \left\langle V_{n} - \Gamma_{m} , \nabla_{\overline{m}_{n}}G \left( V_{n} \right) - \nabla_{\overline{m}_{n}}G\left( \Gamma_{m} \right) \right\rangle_{F}  \\
\notag & + 28\gamma_{n}^{2} +\frac{\gamma_{n}}{\alpha_{n}}\left\|  r_{n} \right\|_{F}^{2} .
\end{align}

Moreover, applying Lemma \ref{lem3maj} and Theorem 5.1 in \cite{godichon2015}, we get for all positive constant $\delta$,
\[
\left\| r_{n} \right\|_{F}^{2} = O \left( \left\| \overline{m}_{n} - m \right\|^{2} \right) = O \left( \frac{\left( \ln n \right)^{1+\delta}}{n} \right) \quad a.s.
\]
Thus, since $2\gamma_{n} \left\langle V_{n} - \Gamma_{m} , \nabla_{\overline{m}_{n}}G \left( V_{n} \right) - \nabla_{\overline{m}_{n}}G\left( \Gamma_{m} \right) \right\rangle_{F} \geq 0$, the  Robbins-Siegmund Theorem (see \cite{Duf97} for instance) ensures that  $\left\| V_{n} - \Gamma_{m} \right\|_{F}$ converges almost surely to a finite random variable and
\[
\sum_{n \geq 1} \gamma_{n}\left\langle V_{n} - \Gamma_{m} , \nabla_{\overline{m}_{n}}G \left( V_{n} \right) - \nabla_{\overline{m}_{n}}G\left( \Gamma_{m} \right) \right\rangle_{F} < + \infty \quad a.s.
\]
Furthermore, by induction, inequality (\ref{majvitps}) becomes
\begin{align*}
\mathbb{E}\left[ \left\| V_{n+1} - \Gamma_{m} \right\|_{F}^{2} \right] & \leq  \left( \prod_{k=1}^{\infty} \left( 1+ \gamma_{k}\alpha_{k} \right)\right) \mathbb{E}\left[ \left\| V_{1} - \Gamma_{m} \right\|_{F}^{2} \right] + 28\left( \prod_{k=1}^{\infty} \left( 1+ \gamma_{k}\alpha_{k} \right) \right)\sum_{k=1}^{\infty}\gamma_{k}^{2} \\
&  + \left( \prod_{k=1}^{\infty} \left( 1+ \gamma_{k}\alpha_{k} \right) \right)  \sum_{k=1}^{\infty} \frac{\gamma_{k}}{\alpha_{k}}\mathbb{E}\left[ \left\| r_{k} \right\|_{F}^{2} \right] . 
\end{align*}
Since $\beta < \alpha$, applying Theorem 4.2 in \cite{godichon2015} and Lemma 6.1, there is a positive constant $C_{0}$ such that
\[
\sum_{k=1}^{\infty}\frac{\gamma_{k}}{\alpha_{k}}\mathbb{E}\left[ \left\| r_{k} \right\|_{F}^{2} \right] = C_{0} \sum_{k=1}^{\infty}k^{-\alpha -1 -\beta} < +\infty . 
\]
Thus, there is a positive constant $M$ such that for all $n \geq 1$, $\mathbb{E}\left[ \left\| V_{n} - \Gamma_{m} \right\|_{F}^{2} \right] \leq M$.
Since $\overline{m}_{n}$ converges almost surely to $m$, one can conclude the proof of the almost sure consistency of $V_n$ with the same arguments as in the proof of Theorem 3.1 in \cite{CCZ11} and the convexity properties given in the Section B of the supplementary file.

Finally, the almost sure consistency of $\overline{V}_n$ is obtained by a direct application of  Topelitz's lemma (see {\it e.g.} Lemma 2.2.13 in \cite{Duf97}).
%\end{proof}

%Then, as in \cite{CCG2015} and \cite{godichon2015}, t
\subsection{Proof of Theorem \ref{theol2l4}}

The proof of Theorem \ref{theol2l4} relies on  properties of  the $p$-th moments of $V_{n}$ for all $p \geq 1$ given in the following three Lemmas.
 These properties enable us, with the application of Markov's inequality, to control the probability of the deviations of  the Robbins Monro  algorithm from $\Gamma_{m}$. 
% The proof of the Lemmas can be found in the Appendix.
 
\begin{lem}\label{lemmajordre}
Under assumptions 1-3(b), for all integer $p$, there is a positive constant $M_{p}$ such that for all $n \geq 1$,
\begin{align*}
\mathbb{E}\left[ \left\| V_{n} - \Gamma_{m} \right\|_{F}^{2p}\right] &\leq M_{p}.
\end{align*}
\end{lem}

%The following lemma gives a bound of the quadratic mean error.
\begin{lem}\label{lem1}
Under assumptions 1-3(b), there are positive constants $C_{1},C_{1}',C_{2},C_{3}$ such that for all $n \geq 1$,
\begin{align*}
\mathbb{E}\left[ \left\| V_{n} - \Gamma_{m} \right\|^{2} \right] &\leq C_{1}e^{-C_{1}'n^{1-\alpha}} + \frac{C_{2}}{n^{\alpha}}  + C_{3}\sup_{E (n/2)+1 \leq k \leq n-1}\mathbb{E}\left[ \left\| V_{k} - \Gamma_{m} \right\|^{4}\right] , 
\end{align*}
where $E(x)$ is the integer part of the real number $x$.
\end{lem}

%The following lemma gives an induction relation for the $4$-th moment of the error.
\begin{lem}\label{lem2}
Under assumptions 1-3(b), for all integer $p' \geq 1$, there are a rank $n_{p'}$ and positive constants $C_{1,p'},C_{2,p'},C_{3,p'},c_{p'}$ such that for all $n \geq n_{p'}$,
\begin{align*}
\mathbb{E}\left[ \left\| V_{n+1} - \Gamma_{m}\right\|_{F}^{4} \right] &\leq \left( 1-c_{p'}\gamma_{n}n^{-\frac{1-\alpha}{p'}}\right)\mathbb{E}\left[ \left\| V_{n} - \Gamma_{m} \right\|_{F}^{4}\right] + \frac{C_{1,p'}}{n^{3\alpha}} + \frac{C_{2,p'}}{n^{2\alpha}}\mathbb{E}\left[ \left\| V_{n} - \Gamma_{m}\right\|_{F}^{2}\right] + \frac{C_{3,p'}}{n^{3\alpha -3\frac{1-\alpha}{p'}}}.
\end{align*}
\end{lem}

%\medskip

We can now prove Theorem \ref{theol2l4}.

%\begin{proof}[Proof of Theorem \ref{theol2l4}]
Let us choose an integer $p'$ such that $p' > 3/2$. Thus, $2~+~\alpha ~ -~3\frac{1-\alpha}{p'}~\geq ~3\alpha$, and applying Lemma \ref{lem2}, there are positive constants $C_{1,p'},C_{2,p'},c_{p'}$ and a rank $n_{p'}$ such that for all $n \geq n_{p'}$,
\begin{equation}
\mathbb{E}\left[ \left\| V_{n+1} - \Gamma_{m}\right\|_{F}^{4} \right] \leq \left( 1-c_{p'}\gamma_{n}n^{-\frac{1-\alpha}{p'}}\right)\mathbb{E}\left[ \left\| V_{n} - \Gamma_{m} \right\|_{F}^{4}\right] + \frac{C_{1,p'}}{n^{3\alpha}} + \frac{C_{2,p'}}{n^{2\alpha}}\mathbb{E}\left[ \left\| V_{n} - \Gamma_{m}\right\|_{F}^{2}\right] .
\end{equation}

%In order to simplify the presentation of the end of the demonstration, we now introduce some notations and results. 
Let us now choose $\beta \in (\alpha , 2\alpha)$ and  $p'$ such that $p' > \frac{1-\alpha}{2\alpha - \beta}$. Note that $ 3\alpha - \beta > \alpha + \frac{1-\alpha}{p'}$.  One can check that there is a rank $n_{p'}' \geq n_{p'}$ such that for all $n \geq n_{p'}'$,
\begin{align*}
(n+1)^{\alpha}C_{1}e^{-C_{1}'n^{1-\alpha}} + \frac{1}{2} + C_{3}2^{\beta +1}\frac{1}{(n+1)^{\beta - \alpha}} & \leq 1 , \\
\left( 1-c_{p'}\gamma_{n}n^{-\frac{1-\alpha}{p'}} \right) \left( \frac{n+1}{n}\right)^{\beta} + 2^{3\alpha}\frac{C_{1,p'} + C_{2,p'}}{(n+1)^{3\alpha - \beta}}  & \leq 1 .
\end{align*}
With the help of a strong induction, we are going to prove the announced results, that is to say  that there are positive constants $C_{p'},C_{\beta}$ such that $2C_{p'} \geq C_{\beta} \geq C_{p'} \geq 1$ and $C_{p'} \geq 2^{\alpha +1}C_{2}$ (with $C_{2}$ defined in Lemma \ref{lem1}), such that for all $n \geq 1$, 
\begin{align*}
\mathbb{E}\left[ \left\| V_{n} - \Gamma_{m} \right\|_{F}^{2}\right] & \leq \frac{C_{p'}}{n^{\alpha}} , \\
\mathbb{E}\left[ \left\| V_{n} - \Gamma_{m}\right\|_{F}^{4} \right] & \leq \frac{C_{\beta}}{n^{\beta}} .
\end{align*}

First, let us choose $C_{p'}$ and $C_{\beta}$ such that 
\begin{align*}
C_{p'} & \geq \max_{k \leq n_{p'}'}\left\lbrace k^{\alpha}\mathbb{E}\left[ \left\| V_{k} - \Gamma_{m} \right\|_{F}^{2}\right] \right\rbrace , \\
C_{\beta} &  \geq \max_{k \leq n_{p'}'}\left\lbrace k^{\beta}\mathbb{E}\left[ \left\| V_{n_{p'}'} - \Gamma_{m} \right\|_{F}^{4}\right] \right\rbrace .
\end{align*}
Thus, for all $k \leq n_{p'}'$,
\begin{align*}
\mathbb{E}\left[ \left\| V_{k} - \Gamma_{m} \right\|_{F}^{2}\right] & \leq \frac{C_{p'}}{k^{\alpha}} , \\
\mathbb{E}\left[ \left\| V_{k} - \Gamma_{m}\right\|_{F}^{4} \right] & \leq \frac{C_{\beta}}{k^{\beta}} .
\end{align*}
We suppose from now that $n \geq n_{p'}'$ and that previous inequalities are verified for all $k \leq n-1$. Applying Lemma \ref{lemmajordre} and by induction,
\begin{align*}
\mathbb{E}\left[ \left\| V_{n+1} - \Gamma_{m}\right\|_{F}^{2}\right] & \leq C_{1}e^{-C_{1}'n^{1-\alpha}} + \frac{C_{2}}{n^{\alpha}} + C_{3}\sup_{E((n+1)/2) +1 \leq k \leq n}\left\lbrace\mathbb{E}\left[ \left\| V_{k} - \Gamma_{m}\right\|_{F}^{4}\right] \right\rbrace \\
& \leq C_{1}e^{-C_{1}'n^{1-\alpha}} + \frac{C_{2}}{n^{\alpha}} + C_{3}\sup_{E((n+1)/2 ) +1 \leq k \leq n}\left\lbrace \frac{C_{\beta}}{k^{\beta}} \right\rbrace \\
& \leq C_{1}e^{-C_{1}'n^{1-\alpha}} + \frac{C_{2}}{n^{\alpha}} + C_{3}2^{\beta}\frac{C_{\beta}}{n^{\beta}}.
 \end{align*}
Since $2C_{p'} \geq C_{\beta} \geq C_{p'} \geq 1$ and since $C_{p'} \geq 2^{\alpha +1}C_{2}$, factorizing by $\frac{C_{p'}}{(n+1)^{\alpha}}$,
 \begin{align*}
\mathbb{E}\left[ \left\| V_{n+1} - \Gamma_{m}\right\|_{F}^{2}\right] &  \leq  C_{p'}C_{1}e^{-C_{1}'n^{1-\alpha}} + C_{p'}2^{-\alpha -1}\frac{1}{n^{\alpha}} + C_{3}2^{\beta}\frac{2C_{p'}}{n^{\beta}} \\
& \leq  \frac{C_{p}'}{(n+1)^{\alpha}}(n+1)^{\alpha}C_{1}e^{-C_{1}'n^{1-\alpha}} +  2^{-\alpha}\left(\frac{n}{n+1}\right)^{\alpha}\frac{C_{p'}}{2(n+1)^{\alpha}} + \frac{C_{3}2^{\beta +1}}{(n+1)^{\beta - \alpha}}\frac{C_{p'}}{(n+1)^{\alpha}} \\
& \leq \frac{C_{p}'}{(n+1)^{\alpha}}C_{1}(n+1)^{\alpha}e^{-C_{1}'n^{1-\alpha}} +  \frac{1}{2}\frac{C_{p'}}{(n+1)^{\alpha}} + C_{3}2^{\beta +1} \frac{1}{(n+1)^{\beta -\alpha}} \frac{C_{p'}}{(n+1)^{\alpha}} \\
& \leq \left( (n+1)^{\alpha}C_{1}e^{-C_{1}'n^{1-\alpha}} + \frac{1}{2} + C_{3}2^{\beta +1}\frac{1}{(n+1)^{\beta - \alpha}} \right) \frac{C_{p'}}{(n+1)^{\alpha}} .
\end{align*}
By definition of $n_{p'}'$, 
\begin{equation}
\mathbb{E}\left[ \left\| V_{n+1} - \Gamma_{m}\right\|_{F}^{2}\right] \leq \frac{C_{p'}}{(n+1)^{\alpha}}.
\end{equation}
In the same way, applying Lemma \ref{lem2} and by induction,
\begin{align*}
\mathbb{E}\left[ \left\| V_{n+1} - \Gamma_{m}\right\|_{F}^{4}\right]  & \leq \left( 1-c_{p'}\gamma_{n}n^{-\frac{1-\alpha}{p'}} \right) \mathbb{E}\left[ \left\| V_{n} - \Gamma_{m}\right\|_{F}^{4}\right] + \frac{C_{1,p'}}{n^{3\alpha}} + \frac{C_{2,p'}}{n^{2\alpha}}\mathbb{E}\left[ \left\| V_{n} - \Gamma_{m}\right\|_{F}^{2}\right]  \\
& \leq \left( 1-c_{p'}\gamma_{n}n^{-\frac{1-\alpha}{p'}} \right)\frac{C_{\beta}}{n^{\beta}}+ \frac{C_{1,p'}}{n^{3\alpha}} + \frac{C_{2,p'}}{n^{2\alpha}}\frac{C_{p'}}{n^{\alpha}}.
\end{align*}
Since $C_{\beta } \geq C_{p'} \geq 1$, factorizing by $\frac{C_{\beta}}{(n+1)^{\beta}}$,
\begin{align*}
\mathbb{E}\left[ \left\| V_{n+1} - \Gamma_{m}\right\|_{F}^{4}\right] & \leq    \left( 1-c_{p'}\gamma_{n}n^{-\frac{1-\alpha}{p'}} \right)\frac{C_{\beta}}{n^{\beta}}+ \left( C_{1,p'} + C_{2,p'}\right) \frac{C_{\beta}}{n^{3\alpha}} \\
& \leq  \left( 1-c_{p'}\gamma_{n}n^{-\frac{1-\alpha}{p'}} \right) \left( \frac{n+1}{n}\right)^{\beta}\frac{C_{\beta}}{n^{\beta}} + 2^{3\alpha}\frac{C_{1,p'} + C_{2,p'}}{(n+1)^{3\alpha - \beta}}\frac{C_{\beta}}{(n+1)^{\beta}} \\
& \leq \left( \left( 1-c_{p'}\gamma_{n}n^{-\frac{1-\alpha}{p'}} \right) \left( \frac{n+1}{n}\right)^{\beta} + 2^{3\alpha}\frac{C_{1,p'} + C_{2,p'}}{(n+1)^{3\alpha - \beta}} \right) \frac{C_{\beta}}{(n+1)^{\beta}}.
\end{align*}
By definition of $n_{p'}'$, 
\begin{equation}
\mathbb{E}\left[ \left\| V_{n+1} - \Gamma_{m}\right\|_{F}^{4}\right] \leq \frac{C_{\beta}}{(n+1)^{\beta}},
\end{equation}
which concludes the induction and the proof. 
%\end{proof}

\subsection{Proof of Theorem \ref{th:cvgeqm}}

In order to prove Theorem \ref{th:cvgeqm}, we first recall the following Lemma.

\begin{lem}[\cite{godichon2015}]\label{lemsumg}
Let $Y_{1},...,Y_{n}$ be random variables taking values in a normed vector space such that for all positive constant $q$ and for all $k \geq 1$, $\mathbb{E}\left[ \left\| Y_{k} \right\|^{q} \right] < \infty$. Then, for all real numbers $a_{1},...,a_{n}$ and for all integer $p$, we have
\begin{equation}
\mathbb{E}\left[ \left\| \sum_{k=1}^{n} a_{k}Y_{k} \right\|^{p} \right] \leq \left( \sum_{k=1}^{n} \left| a_{k} \right| \left( \mathbb{E}\left[ \left\| Y_{k} \right\|^{p} \right] \right)^{\frac{1}{p}} \right)^{p}
\end{equation}
\end{lem}
We can now prove Theorem \ref{th:cvgeqm}.
%\begin{proof}[Proof of Theorem \ref{th:cvgeqm}]
Let us rewrite decomposition (\ref{decdelta}) as follows
\begin{equation}
\nabla_{m}^{2}G\left( \Gamma_{m} \right) \left( V_{n} - \Gamma_{m} \right) = \frac{T_{n}}{\gamma_{n}} - \frac{T_{n+1}}{\gamma_{n}} + \xi_{n+1} - r_{n} - r_{n}' - \delta_{n},
\end{equation}
with $T_{n} := V_{n} - \Gamma_{m}$. As in \cite{Pel00}, we sum these equalities, apply Abel's transform and divide by $n$ to get
\begin{align*}
\nabla_{m}^{2}G\left( \Gamma_{m} \right) \left( \overline{V}_{n} - \Gamma_{m} \right) &= \frac{1}{n}\left(\frac{T_{1}}{\gamma_{1}} - \frac{T_{n+1}}{\gamma_{n+1}} + \sum_{k=2}^{n} T_{k} \left( \frac{1}{\gamma_{k}} - \frac{1}{\gamma_{k-1}}\right) - \sum_{k=1}^{n} \delta_{k} - \sum_{k=1}^{n} r_{k} - \sum_{k=1}^{n} r_{k}' + \sum_{k=1}^{n} \xi_{k+1}\right).
\end{align*}
We now bound the quadratic mean of each term at the right-hand side of previous equality. First, we have $\frac{1}{n^{2}}\mathbb{E}\left[ \left\| \frac{T_{1}}{\gamma_{1}}\right\|_{F}^{2}\right]= o \left( \frac{1}{n} \right)$. Applying Theorem \ref{theol2l4},
\begin{align*}
\frac{1}{n^{2}}\mathbb{E}\left[ \left\| \frac{T_{n+1}}{\gamma_{n}}\right\|_{F}^{2}\right] & \leq \frac{1}{n^{2}}\frac{C'c_{\gamma}^{-2}}{n^{-\alpha}} = o \left( \frac{1}{n}\right) . 
\end{align*}
Moreover, since $\left| \gamma_{k}^{-1} - \gamma_{k-1}^{-1}\right| \leq 2\alpha c_{\gamma}^{-1}k^{\alpha -1}$, the application of Lemma \ref{lemsumg} and Theorem~\ref{theol2l4} gives
\begin{align*}
\frac{1}{n^{2}}\mathbb{E}\left[ \left\| \sum_{k=2}^{n} \left( \gamma_{k}^{-1} - \gamma_{k-1}^{-1}\right)T_{k} \right\|_{F}^{2} \right] & \leq \frac{1}{n^{2}}\left( \sum_{k=2}^{n} \left| \gamma_{k}^{-1} - \gamma_{k-1}^{-1} \right| \sqrt{\mathbb{E}\left[ \left\| T_{k} \right\|_{F}^{2}\right]} \right)^{2} \\
& \leq \frac{1}{n^{2}}4\alpha^{2}c_{\gamma}^{-2}C'\left( \sum_{k=2}^{n} \frac{1}{k^{1-\alpha /2}} \right)^{2} \\
& = O \left( \frac{1}{n^{2-\alpha}}\right) \\
& = o \left( \frac{1}{n} \right) ,
\end{align*}
since $\alpha < 1$. In the same way, since $\left\| \delta_{n} \right\|_{F} \leq 6C \left\| T_{n} \right\|_{F}^{2}$, applying Lemma \ref{lemsumg} and Theorem~\ref{theol2l4} with $\beta > 1$, 
\begin{align*}
\frac{1}{n^{2}}\mathbb{E}\left[ \left\| \sum_{k=1}^{n} \delta_{k} \right\|_{F}^{2}\right] & \leq \frac{1}{n^{2}}\left( \sum_{k=1}^{n} \sqrt{\mathbb{E}\left[ \left\| \delta_{k} \right\|_{F}^{2}\right]} \right)^{2} \\
& \leq \frac{36C^{2}}{n^{2}}\left( \sum_{k=1}^{n} \sqrt{\mathbb{E}\left[ \left\| T_{k} \right\|_{F}^{4}\right]} \right)^{2} \\
& \leq \frac{36C^{2}C_{\beta}}{n^{2}}\left( \sum_{k=1}^{n} \frac{1}{k^{\beta /2}} \right)^{2} \\
& = O \left( \frac{1}{n^{\beta}}\right) \\
& = o \left( \frac{1}{n}\right) ,
\end{align*}
Moreover, let $D := 12 \left( \sqrt{C} + C \sqrt{\left\| \Gamma_{m}\right\|_{F}} \right)$. Since $\left\| r_{n} \right\|_{F} \leq D \left\| \overline{m}_{n} - m \right\|$, and since there is a positive constant $C''$ such that for all $n\geq 1$, $\mathbb{E}\left[ \left\| \overline{m}_{n}- m \right\|^{2}\right] \leq C''n^{-1}$,
\begin{align*}
\frac{1}{n^{2}}\mathbb{E}\left[ \left\| \sum_{k=1}^{n} r_{k} \right\|_{F}^{2}\right] & \leq \frac{1}{n^{2}}\left( \sum_{k=1}^{n} \sqrt{\mathbb{E}\left[ \left\| r_{k} \right\|_{F}^{2}\right]} \right)^{2} \\
& \leq \frac{D^{2}}{n^{2}}\left( \sum_{k=1}^{n} \sqrt{\mathbb{E}\left[ \left\| \overline{m}_{n}-m \right \|^{2}\right]} \right) \\
& \leq \frac{D^{2}C''}{n^{2}}\left( \sum_{k=1}^{n}\frac{1}{k^{1/2}}\right)^{2} \\
& = O \left( \frac{1}{n}\right) .
\end{align*}
Since $\left\| r_{n}' \right\|_{F} \leq C_{0} \left\| \overline{m}_{n} - m \right\| \left\| V_{n} - \Gamma_{m}\right\|_{F}^{2}$ with $C_{0} := 12\left( C\sqrt{\left\| \Gamma_{m}\right\|_{F}}  + C^{3/4} \right) $,  Cauchy-Schwarz's inequality and Lemma \ref{lemsumg} give
\begin{align*}
\frac{1}{n^{2}}\mathbb{E}\left[ \left\| \sum_{k=1}^{n} r_{n}' \right\|_{F}^{2}\right] & \leq \frac{1}{n^{2}}\left( \sum_{k=1}^{n}\sqrt{\mathbb{E}\left[ \left\| r_{n}' \right\|_{F}^{2}\right]} \right)^{2} \\
& \leq \frac{C_{0}^{2}}{n^{2}}\left( \sum_{k=1}^{n} \sqrt{\mathbb{E}\left[ \left\| \overline{m}_{n} - m \right\|^{2}\left\| V_{n} - \Gamma_{m}\right\|_{F}^{2}\right]}\right)^{2} \\
& \leq \frac{C_{0}^{2}}{n^{2}}\left( \sum_{k=1}^{n} \left( \mathbb{E}\left[ \left\| \overline{m}_{n} - m \right\|^{4}\right] \right)^{\frac{1}{4}}\left( \mathbb{E}\left[ \left\| V_{n} - \Gamma_{m} \right\|_{F}^{4} \right] \right)^{\frac{1}{4}} \right)^{2} .
\end{align*}
Applying Theorem 4.2 in \cite{godichon2015} and Theorem 3.3,
\begin{align*}
\frac{1}{n^{2}}\mathbb{E}\left[ \left\| \sum_{k=1}^{n} r_{n}' \right\|_{F}^{2}\right] & \leq \frac{C_{0}^{2}\sqrt{C_\beta}\sqrt{K_{2}}}{n^{2}}\left( \sum_{k=1}^{n} \frac{1}{k^{\beta /4 + 1 /2}}\right)^{2} \\
& = O \left( \frac{1}{n^{1 + \beta /2}}\right) \\
& = o \left( \frac{1}{n}\right) ,
\end{align*}
since $\beta >0$. Finally, one can easily check that $\mathbb{E}\left[ \left\| \xi_{n+1}\right\|_{F}^{2} \right] \leq 1$, and since $\left( \xi_{n} \right)$ is a sequence of martingale differences adapted to the filtration $\left( \mathcal{F}_{n} \right)$,
\begin{align*}
\frac{1}{n^{2}}\mathbb{E}\left[ \left\| \sum_{k=1}^{n} \xi_{k+1} \right\|_{F}^{2}\right] & = \frac{1}{n^{2}}\left( \sum_{k=1}^{n} \mathbb{E}\left[ \left\| \xi_{k+1}\right\|_{F}^{2}\right] + 2\sum_{k=1}^{n}\sum_{k'=k+1}^{n}\mathbb{E}\left[ \left\langle \xi_{k+1},\xi_{k'+1}\right\rangle_{F} \right] \right) \\
& = \frac{1}{n^{2}}\left( \sum_{k=1}^{n} \mathbb{E}\left[ \left\| \xi_{k+1}\right\|_{F}^{2}\right] + 2\sum_{k=1}^{n}\sum_{k'=k+1}^{n}\mathbb{E}\left[ \left\langle \xi_{k+1},\mathbb{E}\left[ \xi_{k'+1}\Big| \mathcal{F}_{k'}\right] \right\rangle_{F} \right] \right) \\
& = \frac{1}{n^{2}}\sum_{k=1}^{n} \mathbb{E}\left[ \left\| \xi_{k+1} \right\|_{F}^{2}\right] \\
& \leq \frac{1}{n} .
\end{align*}
Thus, there is a positive constant $K$ such that for all $n \geq 1$,
\[
\mathbb{E}\left[ \left\| \nabla_{m}^{2}G \left( \Gamma_{m}\right) \left( \overline{V}_{n} - \Gamma_{m}\right) \right\|_{F}^{2}\right] \leq \frac{K}{n}.
\]
Let $\lambda_{\min}$ be the smallest eigenvalue of $\nabla_{m}^{2}G \left( \Gamma_{m}\right)$. We have, with  Proposition B.1 in the supplementary file, that  $\lambda_{\min}> 0$  and the announced result is proven,
\begin{align*}
\mathbb{E}\left[ \left\| \overline{V}_{n} - \Gamma_{m}\right\|_{F}^{2}\right] &\leq \frac{K}{\lambda_{\min}^{2}n}.
\end{align*}

%\end{proof}

\section{Concluding remarks}

The simulation study and the illustration on real data indicate that performing robust principal components analysis via  the median covariation matrix, which can bring new information compared to classical PCA, is an interesting alternative to more classical robust principal components analysis techniques. The use of recursive algorithms permits to perform robust PCA on very large datasets, in which outlying observations may be hard to detect. Another interest of the use of such sequential  algorithms is that estimation of the median covariation matrix as well as the principal components can be performed online with automatic update at each new observation and without being obliged to store all the data in memory. A simple modification of the averaged stochastic gradient algorithm is proposed that ensures non negativeness of the estimated covariation matrices. This modified algorithms has better performances on our simulated data.

A deeper study of the asymptotic behaviour of the recursive algorithms would certainly deserve further investigations. Proving the asymptotic normality and obtaining the limiting variance of the sequence of estimators $\overline{V}_n$  when $m$ is unknown would be of great interest. This is a challenging issue that is beyond the scope of the paper and would require  to study the joint weak convergence of the two simultaneous recursive averaged estimators of $m$ and $\Gamma_m$.

The use of the MCM could be interesting  to robustify the estimation in many different statistical models, particularly with functional data. For example, it could be employed as an alternative to robust functional projection pursuit in    
 robust functional time series prediction or for  robust estimation in functional linear regression, with the introduction  of the median cross-covariation matrix.

\medskip

\noindent \textbf{Acknowledgements.} We  thank the company M\'ediam\'etrie for allowing us to illustrate our methodologies with their data. We also thank Dr. Peggy C\'enac for a careful reading of the proofs.

%\appendix

\begin{appendix}

\section{Estimating the median covariation matrix with Weiszfeld's algorithm}
Suppose we have a fixed size sample $X_1, \ldots, X_n$  and we want to estimate the geometric median.

The iterative Weiszfeld's algorithm relies on the fact that the solution $m^*_n$ of the following optimization problem 
\begin{align*}
\min_{\mu \in H} & \sum_{i=1}^n \| X_i - \mu \|
\end{align*}
satisfies, when $m_n^* \neq X_i$, for all $i=1, \ldots, n$
\begin{align*}
m_n^*  &= \sum_{i=1}^n w_i \left({m}_n^{*} \right) \ X_i
\end{align*}
where the weights $w_i(x)$ are defined by
\[
w_i(x) = \frac{ \nrm{X_i-x}^{-1}}{\displaystyle \sum_{j=1}^n  \nrm{X_j-x}^{-1}}.
\]

Weiszfeld's algorithm is based on the following iterative scheme.   Consider first a pilot estimator $\widehat{m}^{(0)}$ of $m$. At step $(e)$, a new approximation $\widehat{m}_n^{(e+1)}$ to $m$ is given by
\begin{align}
\widehat{m}_n^{(e+1)} &= \sum_{i=1}^n w_i \left(\widehat{m}_n^{(e)} \right) \ X_i . 
\label{def:weiszfeldMCM}
\end{align}
 The iterative procedure is stopped when $\nrm{\widehat{m}_n^{(e+1)} - \widehat{m}_n^{(e)}} \leq \epsilon$, for some precision $\epsilon$ known in advance. The final value of the algorithm is denoted by $\widehat{m}_n$. %This leads to an efficient  numerical procedure

The estimator of the MCM is computed similarly. Suppose $\widehat{\Gamma}^{(e)}$ has been calculated at step $(e)$, then at step $(e+1)$, the new approximation  $\widehat{\Gamma}^{(e+1)}$ to $\Gamma_m$ is defined by
\begin{align}
\widehat{\Gamma}^{(e+1)}_n &=  \sum_{i=1}^n W_i \left(\widehat{\Gamma}^{(e)}\right) (X_i-\widehat{m}_n)(X_i-\widehat{m}_n)^T. 
\label{def:algoiter}
\end{align}
The procedure is stopped when $\nrm{\widehat{\Gamma}^{(e+1)} - \widehat{\Gamma}^{(e)}}_F \leq \epsilon$, for some precision $\epsilon$ fixed in advance.

Note that by construction, this algorithm leads to an estimated median covariation matrix that is always non negative.
%Using this algorithm can been thought as minimizing the empirical risk related to $G_m(V)$ (see (\ref{def:popriskcov})).

\section{Convexity results}
In this section, we first give and recall some convexity properties of functional $G_{h}$. The following one gives some information on the spectrum of the Hessian of $G$.
\begin{prop}\label{convexity}
Under assumptions 1-3(b), for all $h \in H$ and $V \in \mathcal{S}(H)$, $\mathcal{S}(H)$ admits an orthonormal basis composed of eigenvectors of $\nabla_{h}^{2}G(V)$. Let us denote by $\left\lbrace \lambda_{h,V,i} , i \in \mathbb{N}\right\rbrace$ the set of eigenvalues of $\nabla_{h}^{2}G(V)$. For all $i \in \mathbb{N}$,
\[
0 \leq \lambda_{h,V,i} \leq C.
\] 
Moreover, there is a positive constant $c_{m}$ such that for all $i \in \mathbb{N}$,
\[
0<c_{m} \leq \lambda_{m,\Gamma_{m},i} \leq C.
\]
Finally, by continuity, there are positive constants $\epsilon , \epsilon '$ such that for all $h \in \mathcal{B}\left( m , \epsilon \right)$ and $V \in \mathcal{B}\left( \Gamma_{m}, \epsilon ' \right)$, and for all $i \in \mathbb{N}$,
\[
\frac{1}{2}c_{m} \leq \lambda_{h,V,i} \leq C.
\]
\end{prop}
The proof is  very similar to the one in \cite{CCZ11} and consequently it is not given here. Furthermore, as in \cite{CCG2015}, it ensures the local strong convexity as shown in the  following corollary.

\begin{cor}\label{corforconv}
Under assumptions 1-3(b), for all positive constant $A$, there is a positive constant $c_{A}$ such that for all $V \in \mathcal{B}\left( \Gamma_{m} , A \right)$ and $h \in \mathcal{B}\left( m, \epsilon \right)$,
\[
\left\langle \nabla_{h}G(V) - \nabla_{h}G(\Gamma_{m}) , V - \Gamma_{m} \right\rangle_{H} \geq c_{A} \left\| V - \Gamma_{m} \right\|_{F}^{2}.
\] 
\end{cor}

Finally, the following lemma gives an upper bound on the remainder term in the Taylor's expansion of the gradient.
\begin{lem}\label{lemdelta}
Under assumptions 1-3(b), for all $h \in H$ and $V \in \mathcal{S}(H)$, 
\begin{equation}
\left\| \nabla G_{h}(V) - \nabla G_{h}\left(\Gamma_{m}\right) - \nabla_{h}^{2}G\left( \Gamma_{m} \right)\left( V - \Gamma_{m} \right) \right\|_{F} \leq 6C  \left\| V - \Gamma_{m} \right\|_{F}^{2}. 
\end{equation}
\end{lem}
\begin{proof}[Proof of Lemma \ref{lemdelta}]
Let $\delta_{V,h}:=\nabla G_{h}(V) - \nabla G_{h}\left(\Gamma_{m}\right) - \nabla_{h}^{2}G\left( \Gamma_{m} \right)\left( V - \Gamma_{m} \right)$, since \\
$ \nabla G_{h}(V)~ - ~\nabla G_{h}\left(\Gamma_{m}\right) =~ \int_{0}^{1}\nabla_{h}^{2}G\left(\Gamma_{m}+t\left( V -\Gamma_{m} \right)\right)\left( V  - \Gamma_{m} \right) dt$, we have 
\begin{align*}
\left\| \delta_{V,h} \right\|_{F} & = \left\| \int_{0}^{1} \nabla_{h}^{2}G\left(\Gamma_{m}+t\left( V -\Gamma_{m} \right)\right)\left( (V  - \Gamma_{m} \right) dt - \nabla_{h}^{2}G\left( \Gamma_{m} \right)\left( V - \Gamma_{m} \right) \right\|_{F} \\
& \leq \int_{0}^{1} \left\| \nabla_{h}^{2}G\left(\Gamma_{m}+t\left( V -\Gamma_{m} \right)\right)\left( (V  - \Gamma_{m} \right)  - \nabla_{h}^{2}G\left( \Gamma_{m} \right)\left( V - \Gamma_{m} \right) \right\|_{F} dt .
\end{align*}
As in the proof of Lemma 5.1 in \cite{CCG2015}, under assumptions 1-3(b), one can check that for all $h \in H$, and $t \in [0,1]$,
\[
\left\| \nabla_{h}^{2}G\left(\Gamma_{m}+t\left( V -\Gamma_{m} \right)\right)\left( (V  - \Gamma_{m} \right)  - \nabla_{h}^{2}G\left( \Gamma_{m} \right)\left( V - \Gamma_{m} \right) \right\|_{F} \leq 6C \left\| V - \Gamma_{m}\right\|_{F}^{2} , 
\]
which concludes the proof. 
\end{proof}
\section{Decompositions of the Robbins-Monro algorithm and proof of Lemma 5.1}

Let us recall that the Robbins-Monro algorithm is defined recursively by
\begin{align*}
V_{n+1} & =  V_{n} + \gamma_{n} \frac{\left( X_{n+1} - \overline{m}_{n} \right)\left( X_{n+1} - \overline{m}_{n} \right)^{T}-V_{n}}{\left\| \left( X_{n+1} - \overline{m}_{n} \right)\left( X_{n+1} - \overline{m}_{n} \right)^{T}-V_{n} \right\|_{F}} \\
& = V_{n} - \gamma_{n} U_{n+1},
\end{align*}
with $U_{n+1}:= - \frac{\left( X_{n+1} - \overline{m}_{n} \right)\left( X_{n+1} - \overline{m}_{n} \right)^{T}-V_{n}}{\left\| \left( X_{n+1} - \overline{m}_{n} \right)\left( X_{n+1} - \overline{m}_{n} \right)^{T}-V_{n} \right\|_{F}}$. 
Let us remark that $\xi_{n+1}:= \nabla_{\overline{m}_{n}}G(V_{n}) - U_{n+1}$, $\left( \xi_{n} \right)$ is a sequence of martingale differences adapted to the filtration $\left( \mathcal{F}_{n} \right)$ and the algorithm can be written as follows
\begin{equation}
\label{decxi} V_{n+1} = V_{n} -  \gamma_{n}\left( \nabla G_{\overline{m}_{n}}(V_{n})- \nabla G_{\overline{m}_{n}} ( \Gamma_{m} )\right) + \gamma_{n}\xi_{n+1} - \gamma_{n}r_{n},
\end{equation}
with $r_{n} := \nabla G_{\overline{m}_{n}} ( \Gamma_{m}) - \nabla G_{m} ( \Gamma_{m})$. Finally, let is  consider the following linearization of  the gradient, 
\begin{equation}
\label{decdelta} V_{n+1} - \Gamma_{m} = \left( I_{\mathcal{S}(H)} - \gamma_{n} \nabla_{m}^{2}G(\Gamma_{m}) \right) \left( V_{n} -  \Gamma_{m}\right) + \gamma_{n}\xi_{n+1} - \gamma_{n}r_{n} - \gamma_{n}r_{n}' - \gamma_{n}\delta_{n},
\end{equation}
with
\begin{align*}
 r_{n}' & := \left( \nabla_{\overline{m}_{n}}^{2}G\left( \Gamma_{m}\right) - \nabla_{m}^{2}G\left( \Gamma_{m}\right)\right)\left( V_{n}-\Gamma_{m}\right) , \\
   \delta_{n} & := \nabla G_{\overline{m}_{n}}\left( V_{n} \right) - \nabla G_{\overline{m}_{n}}\left( \Gamma_{m} \right) - \nabla_{\overline{m}_{n}}^{2}G\left( \Gamma_{m}\right) \left( V_{n} - \Gamma_{m}\right) .  
\end{align*}

\begin{proof}[Proof of Lemma 5.1] The bound of $\left\| \delta_{n} \right\|$ is a corollary of Lemma \ref{lemdelta}.

\noindent \textbf{Bounding $\left\| r_{n} \right\|$} 

Let us recall that for all $h \in H$, $Y(h) := \left( X - h \right) \left( X - h \right)^{T}$. We now define for all $h \in H$ the random function $\varphi_{h}: [0,1] \longrightarrow \mathcal{S}(H)$ defined for all $t \in [0,1]$ by
\[
\varphi_{h}(t) := \frac{Y(m+th) - \Gamma_{m}}{\left\| Y(m+th) - \Gamma_{m} \right\|_{F}}.
\]
Note that $r_{n} = \mathbb{E}\left[ \varphi_{\overline{m}_{n} - m}(0) - \varphi_{\overline{m}_{n} - m }(1)\Big| \mathcal{F}_{n}\right]  $. Thus, by dominated convergence, 
\[
\left\| r_{n} \right\|_{F} \leq \sup_{t \in [0,1]}\mathbb{E}\left[ \left\| \varphi_{\overline{m}_{n}-m}'(t) \right\|_{F} \Big| \mathcal{F}_{n} \right] .
\]
Moreover, one can check that for all $h \in H$, 
\begin{align*}
\varphi_{h}'(t) &  = -\frac{h \left( X -m-th \right)^{T}}{\left\| Y(m+th)  - \Gamma_{m}\right\|_{F}} - \frac{\left( X-m-th \right)h^{T}}{\left\| Y(m+th)  - \Gamma_{m} \right\|_{F}} \\
& +  \left\langle Y(m+th) - \Gamma_{m} , h\left( X - m-th \right)^{T}  \right\rangle_{F} \frac{Y(m+th) - \Gamma_{m}}{\left\| Y(m+th) - \Gamma_{m}\right\|_{F}^{3}}  \\
& + \left\langle Y(m+th) - \Gamma_{m} , \left( X - m-th \right)h^{T}  \right\rangle_{F} \frac{Y(m+th) - \Gamma_{m}}{\left\| Y(m+th) - \Gamma_{m}\right\|_{F}^{3}} .   
\end{align*}
We now bound each term on the right-hand side of previous equality. First, applying Cauchy-Schwarz's inequality and using the fact that for all $h,h' \in H$, $\left\| hh'^{T} \right\|_{F} = \left\| h \right\| \left\| h' \right\|$,
\begin{align*}
\mathbb{E}\left[  \frac{\left\| h \left( X -m-th \right)^{T}\right\|_{F}}{\left\| Y(m+th)  - \Gamma_{m}\right\|_{F}}  \right] & \leq \left\| h \right\| \mathbb{E}\left[ \frac{\left\| X-m-th \right\|}{\left\| Y(m+th) - \Gamma_{m} \right\|_{F}}\right] \\
& \leq \left\| h \right\|  \mathbb{E}\left[ \frac{\sqrt{\left\| Y(m+th) \right\|_{F}}}{\left\| Y(m+th) - \Gamma_{m} \right\|_{F}} \right] \\
& \leq \left\| h \right\| \left( \mathbb{E}\left[ \frac{\sqrt{ \left\| \Gamma_{m} \right\|_{F}}}{\left\| Y(m+th) - \Gamma_{m} \right\|_{F}} \right] + \mathbb{E}\left[ \frac{1}{\sqrt{\left\| Y(m+th) - \Gamma_{m} \right\|_{F}}}  \right] \right) . 
\end{align*}
Thus, since $\mathbb{E}\left[ \frac{1}{\left\| Y(m+th) - \Gamma_{m}\right\|_{F}} \right] \leq C$, 
\begin{equation}\label{inequality1}
\mathbb{E}\left[  \frac{\left\| h \left( X -m-th \right)^{T}\right\|_{F}}{\left\| Y(m+th)  - \Gamma_{m}\right\|_{F}}  \right] \leq \left\| h \right\| \left( C\sqrt{ \left\| \Gamma_{m} \right\|_{F}} + \sqrt{C} \right) .
\end{equation}
In the same way,
\begin{equation}\label{inequality2}
\mathbb{E}\left[  \frac{\left\| \left( X -m-th \right)h^{T}\right\|_{F}}{\left\| Y(m+th)  - \Gamma_{m}\right\|_{F}}  \right] \leq \left\| h \right\| \left( C\sqrt{ \left\| \Gamma_{m} \right\|_{F}} + \sqrt{C} \right) .
\end{equation}
Applying Cauchy-Schwarz's inequality,
\begin{align*}
\mathbb{E}\left[ \left|  \left\langle Y(m+th) - \Gamma_{m} , h\left( X - m-th \right)^{T}  \right\rangle_{F}\right| \frac{\left\| Y(m+th) - \Gamma_{m}\right\|_{F}}{\left\| Y(m+th) - \Gamma_{m}\right\|_{F}^{3}} \right]& \leq \mathbb{E}\left[  \frac{\left\| h\left( X - m-th \right)^{T} \right\|_{F}}{\left\| Y(m+th) - \Gamma_{m} \right\|_{F}} \right] \\
& \leq \left\| h \right\| \mathbb{E}\left[ \frac{\left\| X -m -th \right\|}{\left\| Y(m+th) - \Gamma_{m}\right\|_{F}}\right] \\
& \leq \left\| h \right\| \mathbb{E}\left[ \frac{\sqrt{\left\| Y(m+th) \right\|_{F}} }{\left\| Y(m+th) - \Gamma_{m} \right\|_{F}}\right] .
\end{align*}
Thus, since $\mathbb{E}\left[ \frac{1}{\left\| Y(m+th) - \Gamma_{m} \right\|_{F}}\right] \leq C$, and since for all positive constants $a,b$, $\sqrt{a+b}~\leq~ \sqrt{a}~+ ~\sqrt{b}$,
\begin{align*}
\left\| h \right\| \mathbb{E}\left[ \frac{\sqrt{\left\| Y(m+th) \right\|_{F}} }{\left\| Y(m+th) - \Gamma_{m} \right\|_{F}}\right] & \leq \left\| h \right\| \left( \mathbb{E}\left[ \frac{\sqrt{ \left\| \Gamma_{m} \right\|_{F}}}{\left\| Y(m+th) - \Gamma_{m} \right\|_{F}} \right] + \mathbb{E}\left[ \frac{1}{\sqrt{\left\| Y(m+th) - \Gamma_{m} \right\|_{F}}}  \right] \right) \\
& \leq \left\| h \right\| \left( C\sqrt{ \left\| \Gamma_{m} \right\|_{F}} + \sqrt{C} \right) .
\end{align*}
Finally, 
\begin{align}
\label{inequality3}\mathbb{E}\left[ \left| \left\langle Y(m+th) - \Gamma_{m} , h\left( X - m-th \right)^{T}  \right\rangle_{F}\right| \frac{\left\| Y(m+th) - \Gamma_{m}\right\|_{F}}{\left\| Y(m+th) - \Gamma_{m}\right\|_{F}^{3}} \right] & \leq \left\| h \right\| \left( C\sqrt{ \left\| \Gamma_{m} \right\|_{F}} + \sqrt{C} \right) , \\
\label{inequality4}\mathbb{E}\left[ \left| \left\langle Y(m+th) - \Gamma_{m} , \left( X - m-th \right)h^{T}  \right\rangle_{F}\right| \frac{\left\| Y(m+th) - \Gamma_{m}\right\|_{F}}{\left\| Y(m+th) - \Gamma_{m}\right\|_{F}^{3}} \right] & \leq \left\| h \right\| \left( C\sqrt{ \left\| \Gamma_{m} \right\|_{F}} + \sqrt{C} \right) .
\end{align}
Applying inequalities (\ref{inequality1}) to (\ref{inequality4}) with $h =  \overline{m}_{n} - m $,  the announced result is proven,
\[
\left\| r_{n} \right\|_{F} \leq 4 \left( \sqrt{C} + C \sqrt{\left\| \Gamma_{m}\right\|_{F}}\right) \left\| \overline{m}_{n} - m \right\| .
\]

\noindent \textbf{Bounding $\left\| r_{n}' \right\|$}

For all $h \in H$ and $V \in \mathcal{S}(H)$, we define the random function $\varphi_{h,V}:~\left[ 0 , 1 \right] ~\longrightarrow ~ \mathcal{S}(H)$ such that for all $t \in [0,1]$,
\[
\varphi_{h,V}(t):= \frac{1}{\left\| Y(m+th) - \Gamma_{m} \right\|_{F}}\left( I_{\mathcal{S}(H)} - \frac{\left( Y(m+th) - \Gamma_{m} \right) \otimes_{F} \left( Y(m+th) - \Gamma_{m} \right)}{\left\| Y(m+th) - \Gamma_{m} \right\|_{F}^{2}}\right)\left( V \right) .
\]
Note that $r_{n}' = \mathbb{E}\left[ \varphi_{\overline{m}_{n} - m , V_{n} - \Gamma_{m}}(1) - \varphi_{\overline{m}_{n}-m , V_{n} - \Gamma_{m}}(0) \Big| \mathcal{F}_{n} \right]$. By dominated convergence,
\[
\left\| r_{n}' \right\|_{F} \leq \sup_{t\in [0,1]}\mathbb{E}\left[ \left\| \varphi_{\overline{m}_{n} -m , V_{n} - \Gamma_{m}}'(t)\right\|_{F} \Big| \mathcal{F}_{n} \right] .
\]
Moreover, as for the bound of $\left\| r_{n} \right\|$,  one can check, with an application of Cauchy-Schwarz's inequality, that for all $h \in H$, $V \in \mathcal{S}(H)$, and $t \in [0,1]$,
\begin{align*}
\varphi_{h,V}'(t)&  \leq 6 \frac{\left\| Y(m+th) - \Gamma_{m} \right\|_{F}\left\| h^{T}(X-m-th) \right\|_{F}}{\left\| Y(m+th) - \Gamma_{m} \right\|_{F}^{3}}\left\| V \right\|_{F} \\
& + 6 \frac{\left\| Y(m+th) - \Gamma_{m} \right\|_{F}\left\| h (X-m-th)^{T}\right\|_{F}}{\left\| Y(m+th) - \Gamma_{m} \right\|_{F}^{5}}\left\| \left( Y(m+th) - \Gamma_{m} \right) \otimes_{F} \left( Y (m+th) - \Gamma_{m} \right) (V) \right\|_{F} \\
& \leq 12 \frac{\left\| h (X-m-th)^{T}\right\|_{F}}{\left\| Y(m+th) - \Gamma_{m} \right\|_{F}^{2}}\left\| V \right\|_{F}.
\end{align*}
Finally, 
\begin{align}
\notag \mathbb{E}\left[  \frac{\left\| h (X-m-th)^{T}\right\|_{F}}{\left\| Y(m+th) - \Gamma_{m} \right\|_{F}^{2}}\left\| V \right\|_{F} \right] & \leq \mathbb{E}\left[ \frac{\left\| h \right\| \left\| X-m-th \right\|}{\left\| Y(m+th) - \Gamma_{m} \right\|_{F}^{2}}\left\| V \right\|_{F} \right] \\
\notag & \leq \left\| h \right\| \left\| V \right\|_{F} \mathbb{E}\left[ \frac{\sqrt{\left\| \Gamma_{m} \right\|_{F}}}{\left\| Y(m+th) - \Gamma_{m} \right\|_{F}^{2}} \right] \\
\notag & + \left\| h \right\| \left\| V \right\|_{F}\mathbb{E}\left[ \frac{1}{\left\| Y(m+th) - \Gamma_{m} \right\|_{F}^{3/2}}\right] \\
\label{inequality1'} & \leq \left( C\sqrt{\left\| \Gamma_{m}\right\|_{F}} + C^{3/4}\right)\left\| h \right\| \left\| V \right\|_{F}. 
\end{align}
Then the announced result follows from an application of inequality (\ref{inequality1'}) with $h = \overline{m}_{n} - m$ and $V=V_{n} - \Gamma_{m}$,
\[
\left\| r_{n}' \right\| \leq 12\left( C\sqrt{\left\| \Gamma_{m}\right\|_{F}} + C^{3/4}\right)\left\| \overline{m}_{n} - m \right\| \left\| V_{n} - \Gamma_{m} \right\|_{F}.
\]
\end{proof}

\section{Proofs of Lemma 5.2, 5.3 and 5.4}

\begin{proof}[Proof of Lemma 5.2]
Using decomposition (\ref{decxi}),
\begin{align*}
\left\| V_{n+1} - \Gamma_{m}\right\|_{F}^{2} & = \left\| V_{n} - \Gamma_{m} \right\|_{F}^{2} -2 \gamma_{n} \left\langle V_{n} - \Gamma_{m} , \nabla G_{\overline{m}_{n}}(V_{n}) - \nabla G_{\overline{m}_{n}}(\Gamma_{m}) \right\rangle_{F} \\ & +\gamma_{n}^{2}\left\| \nabla G_{\overline{m}_{n}}(V_{n}) - \nabla G_{\overline{m}_{n}}(\Gamma_{m}) \right\|_{F}^{2} \\
& + \gamma_{n}^{2}\left\| \xi_{n+1}\right\|_{F}^{2} + 2 \gamma_{n} \left\langle V_{n} - \Gamma_{m} - \gamma_{n} \left( \nabla G_{\overline{m}_{n}}(V_{n}) - \nabla G_{\overline{m}_{n}}(\Gamma_{m}) \right) , \xi_{n+1} \right\rangle_{F} \\
& + \gamma_{n}^{2}\left\| r_{n} \right\|_{F}^{2} -2 \gamma_{n}\left\langle r_{n} , V_{n} - \Gamma_{m} \right\rangle_{F} -2\gamma_{n}^{2} \left\langle r_{n} , \xi_{n+1} - \nabla G_{\overline{m}_{n}}(V_{n}) + \nabla G_{\overline{m}_{n}}(\Gamma_{m}) \right\rangle_{F} . 
\end{align*}
Note that for all $h \in H$ and $V \in \mathcal{S}(H)$ we have $\left\| \nabla G_{h}(V) \right\|_{F} \leq 1$. Moreover, $\left\| r_{n} \right\|_{F} \leq 2$ and $\left\| \xi_{n+1} \right\|_{F} \leq 2$. Since for all $h \in H$, $G_{h}$ is a convex function, we get with  Cauchy-Schwarz's inequality,
\begin{align}\label{majord2}
\left\| V_{n+1} - \Gamma_{m} \right\|_{F}^{2} a\leq \left\| V_{n} - \Gamma_{m}\right\|_{F}^{2} + 36\gamma_{n}^{2} +2 \gamma_{n} \left\langle \xi_{n+1} , V_{n} - \Gamma_{m}\right\rangle_{F} -2\gamma_{n} \left\langle r_{n} , V_{n} - \Gamma_{m} \right\rangle_{F} . 
\end{align}
 Let $C' := 4\left( \sqrt{C} + C \sqrt{\left\|\Gamma_{m}\right\|_{F}}\right)$, let us recall that $\left\| r_{n} \right\|_{F} \leq C' \left\| \overline{m}_{n} - m \right\|$. We now prove by induction that for all integer $p \geq 1$, there is a positive constant $M_{p}$ such that for all $n \geq 1$, $\mathbb{E}\left[ \left\| V_{n} - \Gamma_{m} \right\|_{F}^{2p}\right] \leq M_{p}$.

\medskip

The case $p=1$ has been studied  in the proof of Theorem 3.2. Let $p \geq 2$ and suppose from now that for all $k \leq p-1$, there is a positive constant $M_{k}$ such that for all $n \geq 1$,
\[
\mathbb{E}\left[ \left\| V_{n} - \Gamma_{m}\right\|_{F}^{2k}\right] \leq M_{k}.
\]

\medskip

\noindent \textbf{Bounding $\mathbb{E}\left[ \left\| V_{n} - \Gamma_{m} \right\|_{F}^{2p}\right]$.}
\newline
Let us apply inequality (\ref{majord2}), for all $p \geq 2$ and use the fact that  $\left( \xi_{n} \right) $ is a sequence of martingales differences adapted to the filtration $\left( \mathcal{F}_{n} \right)$,
%\begin{align}\label{decordp}
\begin{multline}
\mathbb{E}\left[ \left\| V_{n+1} - \Gamma_{m} \right\|_{F}^{2p} \right]  \leq \mathbb{E}\left[ \left( \left\| V_{n} - \Gamma_{m}\right\|_{F}^{2} + 36\gamma_{n}^{2} +2\gamma_{n}\left\| r_{n} \right\|_{F} \left\| V_{n} - \Gamma_{m}\right\|_{F} \right)^{p} \right] \\
%\notag 
 + \sum_{k=2}^{p} \binom{p}{k} \mathbb{E}\left[\left( 2\gamma_{n} \left\langle V_{n} - \Gamma_{m} , \xi_{n+1} \right\rangle_{F} \right)^{k} \left( \left\| V_{n} - \Gamma_{m} \right\|_{F}^{2} + 36 \gamma_{n}^{2} + 2\gamma_{n} \left\| r_{n} \right\|_{F}\left\| V_{n} - \Gamma_{m} \right\|_{F} \right)^{p-k} \right] .
%\end{align}
\label{decordp} 
\end{multline}
Let us denote by $(*)$ the second term on the right-hand side of inequality (\ref{decordp}). Applying Cauchy-Schwarz's inequality and since $\left\| \xi_{n+1} \right\|_{F} \leq 2$,
\begin{align*}
(*)&  =   \sum_{k=2}^{n} \binom{p}{k}\mathbb{E}\left[ \left( 2\gamma_{n} \left\langle V_{n} - \Gamma_{m} , \xi_{n+1} \right\rangle \right)^{k} \left( \left\| V_{n} - \Gamma_{m} \right\|_{F}^{2} + 36 \gamma_{n}^{2} + 2\gamma_{n} \left\| r_{n} \right\|_{F}\left\| V_{n} - \Gamma_{m} \right\|_{F} \right)^{p-k}\right] \\
& \leq \sum_{k=2}^{p}\binom{p}{k}2^{2k}\gamma_{n}^{k}\mathbb{E}\left[\left\| V_{n} - \Gamma_{m} \right\|_{F}^{k}\left( \left\| V_{n} - \Gamma_{m} \right\|_{F}^{2} + 36 \gamma_{n}^{2} + 2\gamma_{n} \left\| r_{n} \right\|_{F}\left\| V_{n} - \Gamma_{m} \right\|_{F} \right)^{p-k}\right] .
\end{align*}
With the help of Lemma \ref{lemtechnique},
\begin{align*}
(*) & \leq \sum_{k=2}^{p}2^{2k}3^{p-k-1}\gamma_{n}^{k} \mathbb{E}\left[ \left\| V_{n} - \Gamma_{m} \right\|_{F}^{2p-k}\right] + \sum_{k=2}^{p}2^{2k}3^{p-k-1}36^{p-k}\gamma_{n}^{2p-k}\mathbb{E}\left[ \left\| V_{n} - \Gamma_{m} \right\|_{F}^{k}\right] \\
& + \sum_{k=2}^{p}2^{p+k}3^{p-k-1}\gamma_{n}^{p}\mathbb{E}\left[ \left\| r_{n} \right\|_{F}^{p-k}\left\| V_{n} - \Gamma_{m}\right\|_{F}^{p} \right] .
\end{align*}
Applying Cauchy-Schwarz's inequality,
\begin{align*}
\sum_{k=2}^{p}2^{2k}3^{p-k-1}\gamma_{n}^{k}\mathbb{E}\left[ \left\| V_{n} - \Gamma_{m} \right\|_{F}^{2p-k} \right] & = \sum_{k=2}^{p}2^{2k}3^{p-k-1}\gamma_{n}^{k} \mathbb{E}\left[ \left\| V_{n} - \Gamma_{m} \right\|_{F}^{p-1}\left\| V_{n} - \Gamma_{m} \right\|_{F}^{p+1-k}\right] \\
& \leq \sum_{k=2}^{p}2^{2k}3^{p-k-1}\gamma_{n}^{k}\sqrt{\mathbb{E}\left[ \left\| V_{n} - \Gamma_{m} \right\|_{F}^{2(p-1)}\right]}\sqrt{\mathbb{E}\left[ \left\| V_{n} - \Gamma_{m} \right\|_{F}^{2(p+1-k)}\right]}. \\
\end{align*}
By induction,
\begin{align}
\notag \sum_{k=2}^{p}2^{2k}3^{p-k-1}\gamma_{n}^{k}\mathbb{E}\left[ \left\| V_{n} - \Gamma_{m} \right\|_{F}^{2p-k} \right] & \leq \sum_{k=2}^{p}2^{2k}3^{p-k-1}\gamma_{n}^{k}\sqrt{M_{p-1}}\sqrt{M_{p+1-k}} \\
& = O \left( \gamma_{n}^{2} \right). \label{eq1}
\end{align}
In the same way, applying Cauchy-Schwarz's inequality and by induction,
\begin{align}
\notag \sum_{k=2}^{p}2^{2k}3^{p-k-1}36^{p-k}\gamma_{n}^{2p-k}\mathbb{E}\left[ \left\| V_{n} - \Gamma_{m} \right\|_{F}^{k} \right] & = \sum_{k=2}^{p}2^{2k}3^{p-k-1}36^{p-k}\gamma_{n}^{2p-k}\mathbb{E}\left[ \left\| V_{n} - \Gamma_{m} \right\|_{F}\left\| V_{n} - \Gamma_{m} \right\|_{F}^{k-1} \right] \\
\notag & \leq \sum_{k=2}^{p}2^{2k}3^{p-k-1}36^{p-k}\gamma_{n}^{2p-k}\sqrt{M_{1}}\sqrt{M_{k-1}} \\
\label{eq2} & = O \left( \gamma_{n}^{2} \right) ,
\end{align}
since $p \geq 2$. Similarly, since $\left\| r_{n} \right\|_{F} \leq 2$ and since $p \geq 2$, applying Cauchy-Schwarz's inequality and by induction,
\begin{align}
\notag \sum_{k=2}^{p}2^{p+k}3^{p-k-1}\gamma_{n}^{p}\mathbb{E}\left[ \left\| r_{n} \right\|_{F}^{p-k} \left\| V_{n} - \Gamma_{m}\right\|_{F}^{p}\right] & \leq \sum_{k=2}^{p}2^{2p}3^{p-k-1}\gamma_{n}^{p}\mathbb{E}\left[ \left\| V_{n} - \Gamma_{m}\right\|_{F}^{p}\right] \\
\notag & \leq \sum_{k=2}^{p}2^{2p}3^{p-k-1}\gamma_{n}^{p}\sqrt{M_{1}}\sqrt{M_{p-1}} \\
\label{eq3} & = O \left( \gamma_{n}^{2}\right) .
\end{align}
Finally, applying inequalities (\ref{eq1}) to (\ref{eq3}), there is a positive constant $A_{1}'$ such that for all $n \geq 1$,
\begin{equation}\label{maj1}
\mathbb{E}\left[ \sum_{k=2}^{p} \binom{p}{k}\left( 2\gamma_{n} \left\langle V_{n} - \Gamma_{m} , \xi_{n+1} \right\rangle_{F} \right)^{k} \left( \left\| V_{n} - \Gamma_{m} \right\|_{F}^{2} + 36 \gamma_{n}^{2} + 2\gamma_{n} \left\| r_{n} \right\|_{F}\left\| V_{n} - \Gamma_{m} \right\|_{F} \right)^{p-k} \right] \leq A_{1}'\gamma_{n}^{2}.
\end{equation}
We now denote by $(**)$ the first term at the right-hand side of inequality (\ref{decordp}). With the help of Lemma \ref{lemtechnique} and applying Cauchy-Schwarz's inequality,
\begin{align*}
(**) & \leq \mathbb{E}\left[ \left\| V_{n} - \Gamma_{m}\right\|_{F}^{2p}\right] + \sum_{k=1}^{p}\binom{p}{k}\mathbb{E}\left[ \left( 36\gamma_{n}^{2} +2\gamma_{n}\left\langle r_{n} , V_{n} - \Gamma_{m} \right\rangle_{F}\right)^{k}\left\| V_{n} - \Gamma_{m}\right\|_{F}^{2p-2k}\right] \\
& \leq \mathbb{E}\left[ \left\| V_{n} - \Gamma_{m}\right\|_{F}^{2p}\right] + \sum_{k=1}^{p}\binom{p}{k}2^{k-1}\mathbb{E}\left[ \left( 36^{k}\gamma_{n}^{2k} + 2^{k}\gamma_{n}^{k}\left\| r_{n} \right\|_{F}^{k} \left\| V_{n} - \Gamma_{m}\right\|_{F}^{k} \right) \left\| V_{n}-\Gamma_{m}\right\|_{F}^{2p-2k} \right] .
\end{align*}
Moreover, let
\begin{align*}
(***) & := \sum_{k=1}^{p}\binom{p}{k}2^{k-1}\mathbb{E}\left[ \left( 36^{k}\gamma_{n}^{2k} + 2^{k}\gamma_{n}^{k}\left\| r_{n} \right\|_{F}^{k} \left\| V_{n} - \Gamma_{m}\right\|_{F}^{k} \right) \left\| V_{n}-\Gamma_{m}\right\|_{F}^{2p-2k} \right] \\
& = \sum_{k=1}^{p}\binom{p}{k}2^{k-1}36^{k}\gamma_{n}^{2k}\mathbb{E}\left[ \left\| V_{n} - \Gamma_{m} \right\|_{F}^{2p-2k} \right] + \sum_{k=1}^{p}\binom{p}{k}2^{2k-1}\gamma_{n}^{k}\mathbb{E}\left[ \left\| r_{n} \right\|_{F}^{k}\left\| V_{n} - \Gamma_{m} \right\|_{F}^{2p-k}\right] .
\end{align*}
By induction,
\begin{align*}
 \sum_{k=1}^{p}\binom{p}{k}2^{k-1}36^{k}\gamma_{n}^{2k}\mathbb{E}\left[ \left\| V_{n} - \Gamma_{m} \right\|_{F}^{2p-2k} \right] & =  \sum_{k=1}^{p}\binom{p}{k}2^{k-1}36^{k}\gamma_{n}^{2k}M_{p-k} \\
 & = O \left( \gamma_{n}^{2} \right) .
\end{align*}
Moreover,
\begin{align*}
\sum_{k=1}^{p}\binom{p}{k}2^{2k-1}\gamma_{n}^{k}\mathbb{E}\left[ \left\| r_{n} \right\|_{F}^{k}\left\| V_{n} - \Gamma_{m} \right\|_{F}^{2p-k}\right]  & = \sum_{k=2}^{p}\binom{p}{k}2^{2k-1}\gamma_{n}^{k}\mathbb{E}\left[ \left\| r_{n} \right\|_{F}^{k}\left\| V_{n} - \Gamma_{m} \right\|_{F}^{2p-k}\right] \\
& + 2p\gamma_{n}\mathbb{E}\left[ \left\| r_{n} \right\|_{F}\left\| V_{n} - \Gamma_{m}\right\|_{F}^{2p-1}\right] .
\end{align*}
Applying Cauchy-Schwarz's inequality and by induction, since $ \left\| r_{n} \right\|_{F} \leq 2$,
\begin{align*}
\sum_{k=2}^{p}\binom{p}{k}2^{2k-1}\gamma_{n}^{k}\mathbb{E}\left[ \left\| r_{n} \right\|_{F}^{k}\left\| V_{n} - \Gamma_{m} \right\|_{F}^{2p-k}\right] & \leq \sum_{k=2}^{p}\binom{p}{k}2^{3k-1}\gamma_{n}^{k}\mathbb{E}\left[ \left\| V_{n} - \Gamma_{m} \right\|_{F}^{2p-k}\right] \\
& \leq \sum_{k=2}^{p}\binom{p}{k}2^{3k-1}\gamma_{n}^{k}\sqrt{M_{p+1-k}}\sqrt{M_{p-1}} \\
& = O \left( \gamma_{n}^{2} \right). 
\end{align*}
Moreover, applying Theorem 4.2 in \cite{godichon2015} and Hölder's inequality, since $\left\| r_{n} \right\|_{F}~ \leq ~ C'~ \left\| \overline{m}_{n} - m \right\|$,
\begin{align*}
2p\gamma_{n} \mathbb{E}\left[ \left\| r_{n} \right\|_{F}\left\| V_{n} - \Gamma_{m} \right\|_{F} ^{2p-1}\right] & \leq 2C' p \gamma_{n} \mathbb{E}\left[ \left\| \overline{m}_{n} - m \right\| \left\| V_{n} - \Gamma_{m}\right\|_{F}^{2p-1}\right] \\
& \leq 2C' p \gamma_{n} \left( \mathbb{E}\left[ \left\| \overline{m}_{n} - m \right\|^{2p}\right]\right)^{\frac{1}{2p}}\left( \mathbb{E}\left[ \left\| V_{n} - \Gamma_{m}\right\|_{F}^{2p}\right]\right)^{\frac{2p-1}{2p}} \\
& \leq 2C'p\gamma_{n} \frac{K_{p}^{\frac{1}{2p}}}{n^{1/2}}\left( \mathbb{E}\left[ \left\| V_{n} - \Gamma_{m}\right\|_{F}^{2p}\right]\right)^{\frac{2p-1}{2p}}.
\end{align*}
Finally,
\begin{align*}
2C'p\gamma_{n} \frac{K_{p}^{\frac{1}{2p}}}{n^{1/2}}\left( \mathbb{E}\left[ \left\| V_{n} - \Gamma_{m}\right\|_{F}^{2p}\right]\right)^{\frac{2p-1}{2p}} & \leq 2C'p\gamma_{n} \frac{K_{p}^{\frac{1}{2p}}}{n^{1/2}} \max\left\lbrace 1, \mathbb{E}\left[ \left\| V_{n} - \Gamma_{m}\right\|_{F}^{2p}\right]\right\rbrace \\
& \leq 2C'p\gamma_{n} \frac{K_{p}^{\frac{1}{2p}}}{n^{1/2}} \left( 1 + \mathbb{E}\left[ \left\| V_{n} - \Gamma_{m}\right\|_{F}^{2p}\right]\right) .
\end{align*}
Thus, there are positive constants $A_{0}'',A_{1}''$ such that
\begin{equation}
\label{maj2} (**)  \leq \left( 1+ A_{0}''\frac{1}{n^{\alpha + 1/2}} \right) \mathbb{E}\left[ \left\| V_{n} - \Gamma_{m} \right\|_{F}^{2p}\right]  + A_{1}''\frac{1}{n^{\alpha + 1/2}} .
\end{equation}

Finally, thanks to inequalities (\ref{maj1}) and (\ref{maj2}), there are positive constants $A_{0}',A_{1}'$ such that
\begin{align*}
\mathbb{E}\left[ \left\| V_{n+1} - \Gamma_{m} \right\|_{F}^{2p}\right] & \leq \left( 1+A_{0}'\frac{1}{n^{\alpha + 1/2}}\right)\mathbb{E}\left[ \left\| V_{n} - \Gamma_{m}\right\|_{F}^{2p}\right] + A_{1}'\frac{1}{n^{\alpha + 1/2}} \\
& \leq \prod_{k=1}^{n}\left( 1+A_{0}'\frac{1}{k^{\alpha + 1/2}}\right)\mathbb{E}\left[ \left\| V_{1} - \Gamma_{m}\right\|_{F}^{2p}\right]+ \sum_{k=1}^{n}\prod_{j=k+1}^{n}\left( 1+A_{0}' \frac{1}{j^{\alpha + 1/2}}\right) A_{1}'\frac{1}{k^{\alpha + 1/2}} \\
& \leq \prod_{k=1}^{\infty}\left( 1+A_{0}'\frac{1}{k^{\alpha + 1/2}}\right)\mathbb{E}\left[ \left\| V_{1} - \Gamma_{m}\right\|_{F}^{2p}\right]+ \prod_{j=1}^{\infty}\left( 1+A_{0}' \frac{1}{j^{\alpha + 1/2}}\right)\sum_{k=1}^{\infty} A_{1}'\frac{1}{k^{\alpha + 1/2}} \\
& \leq M_{p},
\end{align*}  
which concludes the induction and the proof.

\end{proof}

\begin{proof}[ Proof of Lemma 5.3]
Let us define the following linear operators:
\begin{align*}
\alpha_{n} & := I_{\mathcal{S}(H)} - \gamma_{n}\nabla_{m}^{2}G(\Gamma_{m}) , \\
\beta_{n} & := \prod_{k=1}^{n} \alpha_{k} = \prod_{k=1}^{n} \left( I_{\mathcal{S}(H)} - \gamma_{k} \nabla_{m}^{2}G(\Gamma_{m}) \right) , \\
\beta_{0} & := I_{\mathcal{S}(H)} .
\end{align*}
Using decomposition (\ref{decdelta}) and by induction, for all $n \geq 1$,
\begin{equation}
\label{decbeta} V_{n} - \Gamma_{m} = \beta_{n-1} \left( V_{1} - \Gamma_{m} \right) + \beta_{n-1}M_{n} - \beta_{n-1}R_{n} - \beta_{n-1} R_{n}' - \beta_{n-1}\Delta_{n} ,
\end{equation}
with
\begin{align*}
& M_{n} := \sum_{k=1}^{n-1} \gamma_{k}\beta_{k}^{-1}\xi_{k+1} , & R_{n} := \sum_{k=1}^{n-1} \gamma_{k}\beta_{k}^{-1}r_{k} , \\
& R_{n}' := \sum_{k=1}^{n-1} \gamma_{k}\beta_{k}^{-1}r_{k}' , & \Delta_{n} := \sum_{k=1}^{n-1}\gamma_{k}\beta_{k}^{-1}\delta_{k} .
\end{align*}
We now study the asymptotic behavior of the linear operators $\beta_{n}$ and $\beta_{n-1}\beta_{k}^{-1}$. As in \cite{CCZ11}, one can check that there are positive constants $c_{0},c_{1}$ such that for all integers $k,n \geq 1$ with $k \leq n-1$,
\begin{align}\label{majbeta}
& \left\| \beta_{n-1} \right\|_{op} \leq c_{0}e^{-\lambda_{\min}\sum_{k=1}^{n}\gamma_{n}}, & \left\| \beta_{n-1}\beta_{k}^{-1}\right\|_{op} & \leq c_{1}e^{-\lambda_{\min}\sum_{j=k}^{n}\gamma_{j}} ,
\end{align}
where $\| . \|_{op}$ is the usual spectral norm for linear operators. We now bound the quadratic mean of each term in decomposition (\ref{decbeta}). 
\medskip

\noindent \textbf{Step 1: the quasi deterministic term $\beta_{n-1}(V_{1} - \Gamma_{m})$.} 
\newline
Applying inequality (\ref{majbeta}), there is a positive constant $c_{0}'$ such that
\begin{align}\label{eqdec1}
\notag \mathbb{E}\left[ \left\| \beta_{n-1}\left(V_{1} - \Gamma_{m}\right) \right\|_{F}^{2}\right] &  \leq \left\| \beta_{n-1}\right\|_{op}^{2}\mathbb{E}\left[ \left\| V_{1}-\Gamma_{m}\right\|_{F}^{2}\right] \\
\notag & \leq c_{0}e^{-2\lambda_{\min}\sum_{k=1}^{n}\gamma_{n}}\mathbb{E}\left[ \left\| V_{1} - \Gamma_{m} \right\|_{F}^{2}\right] \\
& \leq c_{0}e^{-c_{0}'n^{1-\alpha}}\mathbb{E}\left[ \left\| V_{1}- \Gamma_{m} \right\|_{F}^{2} \right] .
\end{align}
This term converges exponentially fast to $0$. 

\medskip

\noindent \textbf{Step 2: the martingale term $\beta_{n-1}M_{n}$.}
\newline
Since $\left( \xi_{n} \right)$ is a sequence of martingale differences adapted to the filtration $\left( \mathcal{F}_{n} \right)$,
\begin{align*}
\mathbb{E}\left[ \left\| \beta_{n-1}M_{n} \right\|_{F}^{2}\right] & = \sum_{k=1}^{n-1}\mathbb{E}\left[ \left\| \beta_{n-1}\beta_{k}^{-1}\gamma_{k}\xi_{k+1} \right\|_{F}^{2}\right] + 2 \sum_{k=1}^{n-1}\sum_{k'=k+1}^{n-1}\gamma_{k}\gamma_{k'}\mathbb{E}\left[ \left\langle \beta_{n-1}\beta_{k}^{-1}\xi_{k+1} , \beta_{n-1}\beta_{k'}^{-1}\xi_{k'+1} \right\rangle_{F} \right] \\
& = \sum_{k=1}^{n-1}\mathbb{E}\left[ \left\| \beta_{n-1}\beta_{k}^{-1}\gamma_{k}\xi_{k+1} \right\|_{F}^{2}\right] + 2 \sum_{k=1}^{n-1}\sum_{k'=k+1}^{n-1}\gamma_{k}\gamma_{k'}\mathbb{E}\left[ \left\langle \beta_{n-1}\beta_{k}^{-1}\xi_{k+1} , \beta_{n-1}\beta_{k'}^{-1}\mathbb{E}\left[\xi_{k'+1}|\mathcal{F}_{k'}\right] \right\rangle_{F} \right] \\
& = \sum_{k=1}^{n-1}\mathbb{E}\left[ \left\| \beta_{n-1}\beta_{k}^{-1}\gamma_{k}\xi_{k+1} \right\|_{F}^{2}\right] .
\end{align*}
Moreover, as in \cite{CCG2015}, Lemma \ref{sumexp} ensures that there is a positive constant $C_{1}'$ such that for all $n \geq 1$,
\begin{equation}
\label{eqdec2} \mathbb{E}\left[ \left\| \beta_{n-1}M_{n} \right\|_{F}^{2}\right] \leq \frac{C_{1}'}{n^{\alpha}}.
\end{equation}

\medskip

\noindent \textbf{Step 3: the first remainder term $\beta_{n-1}R_{n}$.}
\newline
Remarking that $\left\| r_{n} \right\|_{F} \leq 4 \left( \sqrt{C} + C\sqrt{\left\|\Gamma_{m}\right\|_{F}} \right) \left\| \overline{m}_{n} - m \right\|$, 
\begin{align*}
\mathbb{E}\left[ \left\| \beta_{n-1}R_{n} \right\|_{F}^{2} \right] & \leq \mathbb{E}\left[ \left( \sum_{k=1}^{n-1} \gamma_{k} \left\| \beta_{n-1}\beta_{k}^{-1} \right\|_{op}\left\| r_{k} \right\|_{F} \right)^{2}\right] \\
& \leq 16 \left( \sqrt{C} + \sqrt{\left\| \Gamma_{m}\right\|_{F}}\right)^{2} \mathbb{E}\left[ \left( \sum_{k=1}^{n-1} \gamma_{k}\left\| \beta_{n-1}\beta_{k}^{-1}\right\|_{op} \left\| \overline{m}_{k}-m \right\| \right)^{2}\right] .
\end{align*}
Applying Lemma 4.3 and Theorem 4.2 in \cite{godichon2015},
\begin{align*}
\mathbb{E}\left[ \left\| \beta_{n-1}R_{n} \right\|_{F}^{2} \right] & \leq 16 \left( \sqrt{C} + C\sqrt{\left\| \Gamma_{m}\right\|_{F}} \right)^{2}\left( \sum_{k=1}^{n-1}\gamma_{k}\left\| \beta_{n-1}\beta_{k}^{-1}\right\|_{op}\sqrt{\mathbb{E}\left[ \left\| \overline{m}_{k} - m \right\|^{2}\right]}\right)^{2} \\
& \leq  16 \left( \sqrt{C} + C\sqrt{\left\| \Gamma_{m}\right\|_{F}} \right)^{2}K_{1}\left( \sum_{k=1}^{n-1}\gamma_{k}\left\| \beta_{n-1}\beta_{k}^{-1}\right\|_{op}\frac{1}{k^{1/2}}\right)^{2}.
\end{align*}
Applying inequality (\ref{majbeta}),
\begin{align*}
\mathbb{E}\left[ \left\| \beta_{n-1}R_{n} \right\|_{F}^{2}\right] & \leq  16 \left( \sqrt{C} + C\sqrt{\Gamma_{m}} \right)^{2}K_{1}\left( \sum_{k=1}^{n-1}\gamma_{k}e^{-\sum_{j=k}^{n}\gamma_{j}}\frac{1}{k^{1/2}}\right)^{2} \\
& \leq   16 \left( \sqrt{C} + C\sqrt{\Gamma_{m}} \right)^{2}K_{1}\left( \sum_{k=1}^{n}\gamma_{k}e^{-\sum_{j=k}^{n}\gamma_{j}}\frac{1}{k^{1/2}}\right)^{2}.
\end{align*}
Splitting the sum into two parts and applying Lemma \ref{sumexp}, we have
\begin{align*}
\mathbb{E}\left[ \left\| \beta_{n-1}R_{n} \right\|_{F}^{2}\right] & \leq 32 \left( \sqrt{C} + C\sqrt{\left\| \Gamma_{m}\right\|_{F}} \right)^{2}K_{1}\left( \sum_{k=1}^{E (n/2)}\gamma_{k}e^{-\sum_{j=k}^{n}\gamma_{j}}\frac{1}{k^{1/2}}\right)^{2} \\
& + 32 \left( \sqrt{C} + C\sqrt{\left\| \Gamma_{m}\right\|_{F}} \right)^{2}K_{1}\left( \sum_{k=E(n/2) +1}^{n}\gamma_{k}e^{-\sum_{j=k}^{n}\gamma_{j}}\frac{1}{k^{1/2}}\right)^{2} \\
& = O \left( \frac{1}{n}\right) .
\end{align*}
Thus, there is a positive constant $C_{2}'$ such that for all $n \geq 1$,
\begin{equation}\label{eqdec3}
\mathbb{E}\left[ \left\| \beta_{n-1}R_{n} \right\|_{F}^{2}\right] \leq  \frac{C_{2}'}{n}.
\end{equation}

\medskip

\noindent \textbf{Step 4: the second remainder term $\beta_{n-1}R_{n}'$.}
\newline
Let us recall that for all $n \geq 1$, $\left\| r_{n}' \right\|_{F}~\leq~12D\left\|  \overline{m}_{n}-m  \right\|\left\| V_{n} - \Gamma_{m} \right\|_{F}$ with $D:= C \sqrt{ \left\| \Gamma_{m} \right\|_{F}} ~+ ~C ^{3/4}$. Thus,
\begin{align*}
\mathbb{E}\left[ \left\| \beta_{n-1}R_{n}' \right\|_{F}^{2}\right] & \leq \mathbb{E}\left[ \left( \sum_{k=1}^{n-1}\gamma_{k}\left\| \beta_{n-1}\beta_{k}^{-1}\right\|_{op}\left\| r_{k}' \right\|_{F} \right)^{2}\right] \\
& \leq 144D^{2}\mathbb{E}\left[ \left( \sum_{k=1}^{n-1}\gamma_{k}\left\| \beta_{n-1}\beta_{k}^{-1}\right\|_{op}\left\| \overline{m}_{k} - m \right\| \left\| V_{k} - \Gamma_{m}\right\|_{F} \right)^{2}\right] .
\end{align*}
Applying Lemma 4.3 in \cite{godichon2015},
\[
\mathbb{E}\left[ \left\| \beta_{n-1}R_{n}' \right\|_{F}^{2}\right]  \leq 144D^{2}\left( \sum_{k=1}^{n-1}\gamma_{k}\left\| \beta_{n-1}\beta_{k}^{-1}\right\|_{op}\sqrt{\mathbb{E}\left[ \left\| \overline{m}_{k} - m \right\|^{2} \left\| V_{k} - \Gamma_{m} \right\|_{F}^{2} \right]}\right)^{2}.
\]
Thanks to Lemma 5.2, there is a positive constant $M_{2}$ such that for all $n \geq 1$, $\mathbb{E}\left[ \left\| V_{n} - \Gamma_{m} \right\|_{F}^{4}\right]~\leq~M_{2}$. Thus, applying Cauchy-Schwarz's inequality and Theorem 4.2 in \cite{godichon2015},
 \begin{align*}
\mathbb{E}\left[ \left\| \beta_{n-1}R_{n}' \right\|_{F}^{2}\right] & \leq 144D^{2}\left( \sum_{k=1}^{n-1}\gamma_{k}\left\| \beta_{n-1}\beta_{k}^{-1}\right\|_{op}\left(\mathbb{E}\left[ \left\| \overline{m}_{k} - m \right\|^{4}\right]\right)^{\frac{1}{4}} \left( \mathbb{E}\left[ \left\| V_{k} - \Gamma_{m} \right\|_{F}^{4} \right]\right)^{\frac{1}{4}}\right)^{2} \\
& \leq 144D^{2}\sqrt{M_{2}K_{2}}\left( \sum_{k=1}^{n-1}\gamma_{k}\left\| \beta_{n-1}\beta_{k}^{-1}\right\|_{op} \frac{1}{k^{1/2}} \right)^{2}.
 \end{align*}
As in step 3, splitting the sum into two parts, one can check that there is a positive constant $C_{1}''$ such that for all $n \geq 1$,
\begin{equation}\label{eqdec4}
\mathbb{E}\left[ \left\| \beta_{n-1}R_{n}' \right\|_{F}^{2}\right] \leq \frac{C_{1}''}{n}.
\end{equation}

\medskip

\noindent \textbf{Step 5: the third remainder term: $\beta_{n-1}\Delta_{n}$}
\newline
Since $\left\| \delta_{n} \right\|_{F} \leq 6C \left\| V_{n} - \Gamma_{m} \right\|_{F}^{2}$, applying Lemma 4.3 in \cite{godichon2015}, 
\begin{align*}
\mathbb{E}\left[ \left\| \beta_{n-1}\Delta_{n} \right\|_{F}^{2}\right] & \leq \mathbb{E}\left[ \left( \sum_{k=1}^{n-1}\gamma_{k}\left\| \beta_{n-1}\beta_{k}^{-1}\right\|_{op}\left\| \delta_{k} \right\|_{F}\right)^{2}\right] \\
& \leq 36C^{2}\mathbb{E}\left[ \left( \sum_{k=1}^{n-1}\gamma_{k}\left\| \beta_{n-1}\beta_{k}^{-1}\right\|_{op}\left\| V_{k} - \Gamma_{m} \right\|_{F}^{2}\right)^{2}\right] \\
& \leq 36C^{2} \left( \sum_{k=1}^{n-1} \gamma_{k}\left\| \beta_{n-1}\beta_{k}^{-1}\right\|_{op}\sqrt{\mathbb{E}\left[ \left\| V_{k} - \Gamma_{m}\right\|_{F}^{4}\right]} \right)^{2} .
\end{align*}
Thanks to Lemma 5.2, there is a positive constant $M_{2}$ such that for all $n \geq 1$, $\mathbb{E}\left[ \left\| V_{n} - \Gamma_{m}\right\|_{F}^{4}\right]~\leq~M_{2}$. Thus, splitting the sum into two parts and applying inequalities (\ref{majbeta}) and Lemma \ref{sumexp}, there are positive constant $c_{0}' , C_{2}'$ such that for all $n \geq 1$,
\begin{align*}
\mathbb{E}\left[ \left\| \beta_{n-1}\Delta_{n} \right\|_{F}^{2}\right] & \leq 72C^{2}M_{2}^{2}\left( \sum_{k=1}^{E(n/2)}\gamma_{k}e^{-\sum_{j=k}^{n}\gamma_{j}} \right)^{2} \\
& + 72C^{2}\sup_{E(n/2)+1 \leq k \leq n-1} \left\lbrace \mathbb{E}\left[ \left\| V_{k} - \Gamma_{m}\right\|_{F}^{4}\right]\right\rbrace \left( \sum_{k=E(n/2)+1}^{n}\gamma_{k}e^{-\sum_{j=k}^{n}\gamma_{j}} \right)^{2} \\
& \leq   C_{2}' \sup_{E(n/2)+1 \leq k \leq n-1}\left\lbrace \mathbb{E}\left[ \left\| V_{k} - \Gamma_{m}\right\|_{F}^{4}\right]\right\rbrace  + O \left( e^{-2c_{0}'n^{1-\alpha}} \right)  .
\end{align*}
Thus, there is a positive constant $C_{0}'$ such that for all $ n \geq 1$,
\begin{equation}\label{eqdec5}
\mathbb{E}\left[ \left\| \beta_{n-1}\Delta_{n}\right\|_{F}^{2}\right] \leq C_{0}'e^{-2c_{0}'n^{1-\alpha}} + C_{2}'\sup_{E(n/2)+1 \leq k \leq n-1}\left\lbrace \mathbb{E}\left[ \left\| V_{k} - \Gamma_{m}\right\|_{F}^{4}\right]\right\rbrace .
\end{equation} 

\medskip

\noindent \textbf{Conclusion:}
\newline
Applying Lemma \ref{lemtechnique} and decomposition (\ref{decbeta}), for all $n \geq 1$,
\begin{align*}
\mathbb{E}\left[ \left\| V_{n} - \Gamma_{m} \right\|_{F}^{2}\right]  & \leq 5  \mathbb{E}\left[ \left\| \beta_{n-1}\left( V_{1}-\Gamma_{m}\right)\right\|_{F}^{2}\right] + 5 \mathbb{E}\left[ \left\| \beta_{n-1}M_{n}\right\|_{F}^{2}\right] + 5 \mathbb{E}\left[ \left\| \beta_{n-1}R_{n} \right\|_{F}^{2}\right]\\
&  + 5\mathbb{E}\left[ \left\| \beta_{n-1}R_{n}' \right\|_{F}^{2}\right] +5 \mathbb{E}\left[ \left\| \beta_{n-1}\Delta_{n}\right\|_{F}^{2}\right]  .
\end{align*}
Applying inequalities (\ref{eqdec1}) to (\ref{eqdec5}), there are positive constants $C_{1},C_{1}',C_{2},C_{3}$ such that for all $n \geq 1$,
\[
\mathbb{E}\left[ \left\| V_{n} - \Gamma_{m} \right\|^{2} \right] \leq C_{1}e^{-C_{1}'n^{1-\alpha}} + \frac{C_{2}}{n^{\alpha}}  + C_{3}\sup_{E(n/2)+1 \leq k \leq n-1}\mathbb{E}\left[ \left\| V_{k} - \Gamma_{m} \right\|_{F}^{4}\right] . 
\]

\end{proof}

\begin{proof}[Proof of Lemma 5.4] Let us define $W_{n}:= V_{n} - \Gamma_{m} -\gamma_{n} \left( \nabla G_{\overline{m}_{n}}(V_{n}) - \nabla G_{\overline{m}_{n}}(\Gamma_{m})\right)$ and use decomposition (\ref{decxi}),
\begin{align*}
\left\| V_{n+1} - \Gamma_{m}\right\|_{F}^{2} &  = \left\| W_{n} \right\|_{F}^{2} + \gamma_{n}^{2}\left\| \xi_{n+1}\right\|_{F}^{2} + \gamma_{n}^{2}\left\| r_{n} \right\|_{F}^{2}    +2 \gamma_{n} \left\langle \xi_{n+1}, V_{n} - \Gamma_{m} \right\rangle_{F} +2\gamma_{n}^{2}\left\langle \xi_{n+1} , \nabla G_{\overline{m}_{n}}(V_{n})\right\rangle_{F}\\
& - 2\gamma_{n}^{2}\left\langle r_{n} ,   \nabla G_{\overline{m}_{n}}(V_{n}) - \nabla G_{\overline{m}_{n}}(\Gamma_{m}) \right\rangle_{F} -2\gamma_{n}\left\langle r_{n} ,V_{n} -\Gamma_{m} \right\rangle_{F}.
\end{align*}
Since $\left\| \xi_{n+1}\right\|_{F} \leq 2$, $\left\| r_{n} \right\|_{F} \leq 2$ and the fact that for all $h \in H$, $V \in \mathcal{S}(H)$, $\nabla_{h}G(V) \leq 1$, we get with an application of Cauchy-Schwarz's inequality
\[ 
\left\| V_{n+1} - \Gamma_{m}\right\|_{F}^{2} \leq \left\| W_{n} \right\|_{F}^{2} +2\gamma_{n} \left\langle \xi_{n+1} , V_{n} - \Gamma_{m}\right\rangle_{F} +2\gamma_{n} \left\| r_{n} \right\|_{F}\left\| V_{n} - \Gamma_{m}\right\|_{F} + 20\gamma_{n}^{2}.
\]
Thus, since $\left( \xi_{n} \right)$ is a sequence of martingale differences adapted to the filtration $\left( \mathcal{F}_{n} \right)$, and since $\left\| W_{n} \right\|_{F}^{2} \leq \left( 1+C^{2}c_{\gamma}^{2}\right) \left\| V_{n} - \Gamma_{m}\right\|_{F}^{2}$ (this inequality follows from  Proposition \ref{convexity} and from the fact that for all $h \in H$, $G_{h}$ is a convex application),
\begin{align*}
\mathbb{E}\left[ \left\| V_{n+1} - \Gamma_{m}\right\|_{F}^{4} \right] & \leq \mathbb{E}\left[ \left\| W_{n} \right\|_{F}^{4}\right] +2\gamma_{n}\mathbb{E}\left[ \left\| r_{n} \right\|_{F} \left\| W_{n} \right\|_{F}^{2}\left\| V_{n} - \Gamma_{m}\right\|_{F}\right]  \\ & + 40 \left( 1+C^{2}c_{\gamma}^{2}\right)\gamma_{n}^{2}\mathbb{E}\left[ \left\| V_{n} - \Gamma_{m}\right\|_{F}^{2}\right] \\
&  + 4\gamma_{n}^{2}\mathbb{E}\left[ \left\langle \xi_{n+1} , V_{n} - \Gamma_{m}\right\rangle_{F}^{2}\right] + 400\gamma_{n}^{4} + 40\gamma_{n}^{3}\mathbb{E}\left[ \left\| r_{n} \right\|_{F} \left\| V_{n} - \Gamma_{m}\right\|_{F}^{2} \right] \\
& + 4\gamma_{n}^{2}\mathbb{E}\left[ \left\| r_{n} \right\|_{F}^{2}\left\| V_{n} - \Gamma_{m}\right\|_{F}^{2}\right] .
\end{align*}
Since $\left\| \xi_{n+1}\right\|_{F} \leq 2$ and $\left\| r_{n} \right\|_{F} \leq 2$, applying Cauchy-Schwarz's inequality, there are positive constants $C_{1}',C_{2}'$ such that for all $n\geq 1$,
\begin{equation}\label{majoord4}
\mathbb{E}\left[ \left\| V_{n+1}-\Gamma_{m}\right\|_{F}^{4}\right] \leq \mathbb{E}\left[ \left\| W_{n} \right\|_{F}^{4}\right] +2\gamma_{n}\mathbb{E}\left[ \left\| r_{n} \right\|_{F} \left\| W_{n} \right\|_{F}^{2}\left\| V_{n} - \Gamma_{m}\right\|_{F}\right] + \frac{C_{1}'}{n^{3\alpha}}+ \frac{C_{2}'}{n^{2\alpha}}\mathbb{E}\left[ \left\| V_{n} - \Gamma_{m}\right\|_{F}^{2}\right].
\end{equation}
We now bound the two first terms at the right-hand side of inequality (\ref{majoord4}).

\medskip

\noindent \textbf{Step 1: bounding $\mathbb{E}\left[ \left\| W_{n} \right\|_{F}^{4}\right]$.}
\newline
Since $\nabla G_{\overline{m}_{n}}(V_{n}) - \nabla G_{\overline{m}_{n}}(\Gamma_{m}) = \int_{0}^{1}\nabla_{\overline{m}_{n}}^{2}G \left( \Gamma_{m} + t \left( V_{n} - \Gamma_{m}\right) \right)\left( V_{n} - \Gamma_{m}\right) dt  $, applying Proposition \ref{convexity}, one can check that
\begin{align*}
 \left\| W_{n} \right\|^{2} &  = \left\| V_{n} - \Gamma_{m}\right\|_{F}^{2} -2\gamma_{n} \left\langle V_{n} - \Gamma_{m} , \nabla G_{\overline{m}_{n}}(V_{n}) - \nabla G_{\overline{m}_{n}}(\Gamma_{m}) \right\rangle_{H} + \gamma_{n}^{2}\left\| \nabla G_{\overline{m}_{n}}(V_{n}) - \nabla G_{\overline{m}_{n}}(\Gamma_{m}) \right\|_{F}^{2} \\
 & \leq \left( 1+ C^{2}\gamma_{n}^{2} \right)\left\| V_{n} - \Gamma_{m}\right\|_{F}^{2} -2 \gamma_{n} \left\langle V_{n} - \Gamma_{m} , \nabla G_{\overline{m}_{n}}(V_{n}) - \nabla G_{\overline{m}_{n}}(\Gamma_{m}) \right\rangle_{H}.
\end{align*}
Since for all $h \in H$, $G_{h}$ is a convex application, $\left\| W_{n} \right\|_{F}^{2} \leq \left( 1+c_{\gamma}^{2}C^{2}\right)\left\| V_{n} - \Gamma_{m}\right\|_{F}^{2}$. Let $p'$ be a positive integer. We now introduce the sequence of events $\left( A_{n,p'}\right)_{n \in \mathbb{N}}$ defined for all $n \geq 1$ by
\begin{equation}
A_{n,p'} := \left\lbrace \omega \in \Omega , \quad \left\| V_{n}(\omega ) - \Gamma_{m}\right\|_{F} \leq n^{\frac{1-\alpha}{p'}} , \quad \text{and } \quad \left\| \overline{m}_{n}(\omega ) - m \right\| \leq \epsilon \right\rbrace ,
\end{equation} 
with $\epsilon$ defined in Proposition \ref{convexity}. For the sake of simplicity, we consider that $\epsilon '$ defined in Proposition \ref{convexity} verifies $\epsilon ' \leq 1$. Applying Proposition \ref{convexity}, let
\begin{align}\label{majbn}
\notag B_{n}:& = \left\langle \nabla G_{\overline{m}_{n}}( V_{n} ) - \nabla G_{\overline{m}_{n}}(\Gamma_{m}) , V_{n} - \Gamma_{m} \right\rangle_{F}\mathbf{1}_{A_{n,p'}}\mathbf{1}_{\left\lbrace \left\| V_{n} - \Gamma_{m}\right\|_{F} \leq \epsilon '\right\rbrace} \\
\notag  & = \int_{0}^{1}\left\langle \nabla_{\overline{m}_{n}}^{2}G\left( \Gamma_{m} + t\left( V_{n} - \Gamma_{m}\right) \right) \left( V_{n} - \Gamma_{m} \right) , V_{n} - \Gamma_{m} \right\rangle_{F}  \mathbf{1}_{\left\lbrace \left\| V_{n} - \Gamma_{m}\right\|_{F} \leq \epsilon '\right\rbrace}\mathbf{1}_{A_{n,p'}} dt \\
& \geq \frac{1}{2} c_{m}\left\| V_{n} - \Gamma_{m} \right\|_{F}^{2} \mathbf{1}_{\left\lbrace \left\| V_{n} - \Gamma_{m}\right\|_{F} \leq \epsilon '\right\rbrace}\mathbf{1}_{A_{n,p'}}  .
\end{align}
In the same way, since $G_{\overline{m}_{n}}$ is convex, let 
\begin{align*}
B_{n}':& = \left\langle \nabla G_{\overline{m}_{n}}( V_{n} ) - \nabla G_{\overline{m}_{n}}(\Gamma_{m}) , V_{n} - \Gamma_{m} \right\rangle_{F}\mathbf{1}_{A_{n,p'}}\mathbf{1}_{\left\lbrace \left\| V_{n} - \Gamma_{m}\right\|_{F} > \epsilon '\right\rbrace} \\
& =  \int_{0}^{1}\left\langle \nabla_{\overline{m}_{n}}^{2}\left( \Gamma_{m} + t\left( V_{n} - \Gamma_{m}\right) \right) \left( V_{n} - \Gamma_{m} \right) , V_{n} - \Gamma_{m} \right\rangle \mathbf{1}_{\left\lbrace \left\| V_{n} - \Gamma_{m}\right\|_{F} > \epsilon '\right\rbrace}\mathbf{1}_{A_{n,p'}}	dt \\
& \geq \int_{0}^{\frac{\epsilon '}{\left\| V_{n} - \Gamma_{m} \right\|_{F}}}\left\langle \nabla_{\overline{m}_{n}}^{2}\left( \Gamma_{m} + t\left( V_{n} - \Gamma_{m}\right) \right) \left( V_{n} - \Gamma_{m} \right) , V_{n} - \Gamma_{m} \right\rangle \mathbf{1}_{\left\lbrace \left\| V_{n} - \Gamma_{m}\right\|_{F} > \epsilon '\right\rbrace} \mathbf{1}_{A_{n,p'}} dt
\end{align*}
Applying Proposition \ref{convexity},
\begin{align}\label{majcn}
\notag B_{n}' & \geq \int_{0}^{\frac{\epsilon '}{\left\| V_{n} - \Gamma_{m} \right\|_{F}}} \frac{1}{2}c_{m} \left\| V_{n} - \Gamma_{m} \right\|_{F}^{2}\mathbf{1}_{\left\lbrace \left\| V_{n} - \Gamma_{m}\right\|_{F} > \epsilon '\right\rbrace} \mathbf{1}_{A_{n,p'}} dt \\
\notag & \geq \frac{\epsilon' c_{m}}{2 \left\| V_{n} - \Gamma_{m} \right\|_{F}}\left\| V_{n} - \Gamma_{m}\right\|_{F}^{2}\mathbf{1}_{\left\lbrace \left\| V_{n} - \Gamma_{m}\right\|_{F} > \epsilon '\right\rbrace} \mathbf{1}_{A_{n,p'}} \\
& \geq \frac{\epsilon'c_{m}}{2}n^{-\frac{1-\alpha}{p'}}\left\| V_{n} - \Gamma_{m}\right\|_{F}^{2}\mathbf{1}_{\left\lbrace \left\| V_{n} - \Gamma_{m}\right\|_{F} > \epsilon '\right\rbrace} \mathbf{1}_{A_{n,p'}}. 
\end{align}
There is a rank $n_{p'}'$ such that for all $n \geq n_{p'}'$, we have $\frac{\epsilon 'c_{m}}{2}n^{-\frac{1-\alpha}{p}}\leq \frac{1}{2}c_{m}$. Thus, applying inequalities (\ref{majbn}) and (\ref{majcn}), for all $n \geq n_{p'}'$,
\[
\left\| W_{n} \right\|_{F}^{2}\mathbf{1}_{A_{n,p'}} \leq \left( 1- \frac{\epsilon 'c_{m} }{2}\gamma_{n}n^{-\frac{1-\alpha}{p'}}\right) \left\| V_{n}  - \Gamma_{m}\right\|_{F}^{2}\mathbf{1}_{A_{n,p'}}.
\]
Thus, there are a positive constant $c_{p'}$ and a rank $n_{p'}$ such that for all $n \geq n_{p'}$,
\begin{align}\label{majwnan}
\notag \mathbb{E}\left[ \left\| W_{n} \right\|_{F}^{4}\mathbf{1}_{A_{n,p'}}\right]  & \leq \left( 1- \frac{\epsilon 'c_{m} }{2}\gamma_{n}n^{-\frac{1-\alpha}{p'}}\right)^{2} \mathbb{E}\left[  \left\| V_{n} - \Gamma_{m}\right\|_{F}^{4}\mathbf{1}_{A_{n,p'}}\right] \\
& \leq \left( 1- 2c_{p'}\gamma_{n}n^{-\frac{1-\alpha}{p'}}\right) \mathbb{E}\left[ \left\| V_{n} - \Gamma_{m}\right\|_{F}^{4}\right] .
\end{align}
Now, we  must get an upper bound for $\mathbb{E}\left[ \left\| W_{n} \right\|_{F}^{4}\mathbf{1}_{A_{n,p'}^{c}}\right]$. Since $\left\| W_{n} \right\|_{F}^{2} \leq \left( 1+c_{\gamma}^{2}C^{2}\right) \left\| V_{n} - \Gamma_{m}\right\|_{F}^{2}$ and since there is a positive constant $c_{0}$ such that for all $n \geq 1$, 
\begin{align*}
\left\| V_{n} - \Gamma_{m}\right\|_{F}  & \leq \left\| V_{1} - \Gamma_{m} \right\|_{F}~+ ~\sum_{k=1}^{n} \gamma_{k} \leq ~c_{0}n^{1-\alpha}
\end{align*}
we have
\begin{align*}
\mathbb{E}\left[ \left\| W_{n} \right\|_{F}^{4}\mathbf{1}_{A_{n,p'}^{c}}\right] & \leq \left( 1+c_{\gamma}^{2}C^{2}\right)^{2}\mathbb{E}\left[ \left\| V_{n} - \Gamma_{m}\right\|_{F}^{4} \mathbf{1}_{A_{n,p'}^{c}}\right] \\ 
& \leq \left( 1+c_{\gamma}^{2}C^{2}\right)^{2}c_{0}^{4}n^{4-4\alpha}\mathbb{P}\left[ A_{n,p'}^{c}\right] \\
& \leq \left( 1+c_{\gamma}^{2}C^{2}\right)^{2}c_{0}^{4}n^{4-4\alpha} \left( \mathbb{P}\left[ \left\| \overline{m}_{n} - m \right\| \geq \epsilon \right] + \mathbb{P}\left[ \left\| V_{n} - \Gamma_{m}\right\|_{F} \geq n^{\frac{1-\alpha}{p'}}\right] \right) .
\end{align*}
Applying Markov's inequality, Theorem 4.2 in \cite{godichon2015} and Lemma 5.2,
\begin{align*}
\mathbb{E}\left[ \left\| W_{n} \right\|_{F}^{4}\mathbf{1}_{A_{n,p'}^{c}}\right] & \leq \left( 1+c_{\gamma}^{2}C^{2}\right)^{2}c_{0}^{4}n^{4-4\alpha} \left( \frac{\mathbb{E}\left[ \left\| \overline{m}_{n}-m \right\|^{2p''}\right]}{\epsilon^{2p''}} + \frac{\mathbb{E}\left[ \left\| V_{n} - \Gamma_{m}\right\|_{F}^{2q}\right]}{n^{2q\frac{1-\alpha}{p'}}} \right) \\
& \leq \frac{K_{p''}}{\epsilon^{2p''}}\left( 1+c_{\gamma}^{2}C^{2}\right)^{2}c_{0}^{4}n^{4-4\alpha - p''} + \left( 1+c_{\gamma}^{2}C^{2}\right)^{2}c_{0}^{4}M_{q}n^{4-4\alpha - 2q \frac{1-\alpha}{p'}}. 
\end{align*}
Taking $p'' \geq  4-\alpha $ and $q \geq p'\frac{4-\alpha}{2(1-\alpha )}$,
\begin{equation}\label{majwnanc}
\mathbb{E}\left[ \left\| W_{n} \right\|_{F}^{4}\mathbf{1}_{A_{n,p'}^{c}} \right] = O \left( \frac{1}{n^{3\alpha}}\right) .
\end{equation}
Thus, applying inequalities (\ref{majwnan}) and (\ref{majwnanc}), there are positive constants $c_{p'}$, $C_{1,p'}$ and a rank $n_{p'}$ such that for all $n \geq n_{p'}$,
\begin{equation}
\label{majwn} \mathbb{E}\left[ \left\| W_{n} \right\|_{F}^{4}\right] \leq \left( 1- 2c_{p'}\gamma_{n}n^{-\frac{1-\alpha}{p'}}\right) \mathbb{E}\left[ \left\| V_{n} - \Gamma_{m} \right\|_{F}^{4}\right] + \frac{C_{1,p'}}{n^{3\alpha}}.
\end{equation}

\medskip

\noindent \textbf{Step 2: bounding $2\gamma_{n}\mathbb{E}\left[ \left\| r_{n} \right\|_{F}\left\| W_{n} \right\|_{F}^{2}\left\| V_{n} - \Gamma_{m}\right\|_{F}\right]$.}
\newline
Since $\left\| W_{n} \right\|_{F}^{2} \leq \left( 1+c_{\gamma}^{2}C^{2}\right) \left\| V_{n} - \Gamma_{m}\right\|_{F}^{2}$, applying Lemma \ref{lemtechnique}, let
\begin{align*}
D_{n} :& =2\gamma_{n}\mathbb{E}\left[ \left\| r_{n} \right\|_{F}\left\| W_{n} \right\|_{F}^{2}\left\| V_{n} - \Gamma_{m}\right\|_{F}\right] \\
& \leq 2\left( 1+c_{\gamma}^{2}C^{2}\right) \gamma_{n} \mathbb{E}\left[ \left\| r_{n} \right\|_{F}\left\| V_{n} - \Gamma_{m}\right\|_{F}^{3}\right] \\
& \leq \frac{2}{c_{p'}}\left( 1+c_{\gamma}^{2}C^{2}\right)^{2}\gamma_{n}n^{\frac{1-\alpha}{p'}}\mathbb{E}\left[ \left\| r_{n} \right\|_{F}^{2}\left\| V_{n} - \Gamma_{m}\right\|_{F}^{2}\right] + \frac{1}{2}c_{p'}\gamma_{n}n^{-\frac{1-\alpha}{p'}}\mathbb{E}\left[ \left\| V_{n} - \Gamma_{m}\right\|_{F}^{4}\right] \\
& \leq \frac{2}{c_{p'}^{2}}\left( 1+c_{\gamma}^{2}C^{2}\right)^{4}\gamma_{n}n^{3\frac{1-\alpha}{p'}}\mathbb{E}\left[ \left\| r_{n} \right\|_{F}^{4}\right] + c_{p'}\gamma_{n}n^{-\frac{1-\alpha}{p'}}\mathbb{E}\left[ \left\| V_{n} - \Gamma_{m}\right\|_{F}^{4}\right] .
\end{align*}
Since $\left\| r_{n} \right\|_{F} \leq \left( \sqrt{C} + C\sqrt{\left\| \Gamma_{m}\right\|_{F}}\right) \left\| \overline{m}_{n} - m \right\|_{F}$ and applying Theorem 4.2 in \cite{godichon2015},
\begin{align}\label{majdn}
\notag D_{n} & \leq \frac{2}{c_{p'}^{2}}\left( 1+c_{\gamma}^{2}C^{2}\right)^{4}\left( \sqrt{C} + C\sqrt{\left\| \Gamma_{m}\right\|_{F}}\right)^{4}\gamma_{n}n^{3\frac{1-\alpha}{p'}}\mathbb{E}\left[ \left\| \overline{m}_{n} - m \right\|^{4}\right] + c_{p'}\gamma_{n}n^{-\frac{1-\alpha}{p'}}\mathbb{E}\left[ \left\| V_{n} - \Gamma_{m}\right\|_{F}^{4}\right] \\
\notag & \leq \frac{2}{c_{p'}^{2}}K_{2}\left( 1+c_{\gamma}^{2}C^{2}\right)^{4}\left( \sqrt{C} + C\sqrt{\left\| \Gamma_{m}\right\|_{F}}\right)^{4}\gamma_{n}n^{3\frac{1-\alpha}{p'}}\frac{1}{n^{2}} + c_{p'}\gamma_{n}n^{-\frac{1-\alpha}{p'}}\mathbb{E}\left[ \left\| V_{n} - \Gamma_{m}\right\|_{F}^{4}\right] \\
& = c_{p'}\gamma_{n}n^{-\frac{1-\alpha}{p'}}\mathbb{E}\left[ \left\| V_{n} - \Gamma_{m}\right\|_{F}^{4}\right] + O \left( \frac{1}{n^{2 + \alpha -3(1-\alpha)/p'}}\right) .
\end{align}

\medskip

\noindent \textbf{Step 3: Conclusion.}
\newline
Applying inequalities (\ref{majoord4}), (\ref{majwn}) and (\ref{majdn}), there are a rank $n_{p'}$ and positive constants $c_{p'}, C_{1,p'},C_{2,p'},C_{3,p'}$ such that for all $n \geq n_{p'}$,
\[
\mathbb{E}\left[ \left\| V_{n+1} - \Gamma_{m}\right\|_{F}^{4} \right] \leq \left( 1- c_{p'}\gamma_{n}n^{-\frac{1-\alpha}{p'}}\right)\mathbb{E}\left[ \left\| V_{n} - \Gamma_{m} \right\|_{F}^{4}\right] + \frac{C_{1,p'}}{n^{3\alpha}} + \frac{C_{2,p'}}{n^{2\alpha}}\mathbb{E}\left[ \left\| V_{n} - \Gamma_{m}\right\|_{F}^{2}\right] + \frac{C_{3,p'}}{n^{2+\alpha -3\frac{1-\alpha}{p'}}}.
\]
\end{proof}

\section{Some technical inequalities}

First, the following lemma recalls some well-known inequalities.
\begin{lem}\label{lemtechnique}
Let $a,b,c$ be positive constants. Then, 
\begin{align*}
ab &  \leq \frac{a^{2}}{2c}+\frac{b^{2}c}{2}, \\
a & \leq \frac{c}{2}+ \frac{a^{2}}{2c}.
\end{align*}
Moreover, let $k,p$ be positive integers and $a_{1},...,a_{p}$ be positive constants. Then,
\[
\left( \sum_{j=1}^{p}a_{j} \right)^{k} \leq p^{k-1}\sum_{j=1}^{p}a_{j}^{k}.
\]
\end{lem}

The following lemma gives the asymptotic behavior for some specific sequences of descent steps.
\begin{lem}\label{sumexp}
Let $\alpha,\beta $ be non-negative constants such that $0<\alpha<1$, and $\left( u_{n}\right)$, $\left( v_{n} \right)$ be two sequences defined for all $n \geq 1$ by
\begin{align*}
u_{n} & := \frac{c_{u}}{n^{\alpha}}, & v_{n}:=\frac{c_{v}}{n^{\beta}},
\end{align*}
with $c_{u},c_{v} > 0$. Thus, there is a positive constant $c_{0}$ such that for all $n \geq 1$,
\begin{align}
\label{sumexp1} & \sum_{k=1}^{E(n/2)} e^{-\sum_{j=k}^{n}u_{j}}u_{k}v_{k} = O \left(  e^{-c_{0}n^{1-\alpha}} \right) , \\
\label{sumexp2} & \sum_{k=E(n/2)+1}^{n}e^{-\sum_{j=k}^{n}u_{j}}u_{k}v_{k}  = O \left( v_{n}\right) ,
\end{align}
where $E(.)$ is the integer part function.
\end{lem}
\begin{proof}[ Proof of Lemma \ref{sumexp}]
We first prove inequality (\ref{sumexp1}). For all $n \geq 1$,
\begin{align*}
\sum_{k=1}^{E(n/2)} e^{-\sum_{j=k}^{n}u_{j}}u_{k}v_{k} & = c_{u}c_{v}\sum_{k=1}^{E(n/2)} e^{-\sum_{j=k}^{n}u_{j}}\frac{1}{k^{\alpha + \beta}} \\
& \leq c_{u}c_{v}\sum_{k=1}^{E(n/2)} e^{-c_{u}\sum_{j=k}^{n}\frac{1}{j^{\alpha}}} .
\end{align*}
Moreover, for all $ k\leq E(n/2)$,
\begin{align*}
 c_{u}\sum_{j=k}^{n}\frac{1}{j^{\alpha}} & \geq  c_{u}\frac{n}{2}\frac{1}{n^{\alpha}} \\
 & \geq \frac{c_{u}}{2}n^{1-\alpha}.
\end{align*}
Thus,
\[
\sum_{k=1}^{E(n/2)} e^{-\sum_{j=k}^{n}u_{j}}u_{k}v_{k} \leq c_{u}c_{v}ne^{-\frac{c_{u}}{2}n^{1-\alpha}}.
\]
We now prove inequality (\ref{sumexp2}). With the help of an integral test for convergence,
\begin{align*}
\sum_{j=k}^{n}u_{j} & = c_{u} \sum_{j=k}^{n}\frac{1}{j^{\alpha}} \\
& \geq c_{u}\int_{k}^{n+1}\frac{1}{t^{\alpha}}dt \\
& \geq \frac{c_{u}}{1-\alpha}\left( (n+1)^{1-\alpha} - k^{-\alpha} \right). 
\end{align*}
Thus,
\[
\sum_{k=E(n/2)+1}^{n}e^{-\sum_{j=k}^{n}u_{j}}u_{k}v_{k} \leq c_{u}c_{v}e^{-(n+1)^{1-\alpha}}\sum_{k=E(n/2)+1}^{n}e^{k^{1-\alpha}}k^{-\alpha - \beta}
\]
With the help of an integral test for convergence, there is a rank $n_{u,v}$ (for sake of simplicity, we consider that $n_{u,v}=1$) such that for all $n \geq n_{u,v}$,
\begin{align*}
\sum_{k=E(n/2)+1}^{n}e^{k^{1-\alpha}}k^{-\alpha - \beta} & \leq \int_{E(n/2)+1}^{n+1}e^{t^{1-\alpha}}t^{-\alpha - \beta }dt \\
& \leq \frac{1}{1-\alpha} \left[ e^{t^{1-\alpha}}t^{-\beta} \right]_{E(n/2)+1}^{n} + \beta\int_{E(n/2)+1}^{n}e^{t^{1-\alpha}}t^{-1-\beta}dt \\
& = e^{(n+1)^{1-\alpha}(n+1)^{-\beta}} + o \left( \int_{E(n/2)+1}^{n+1}e^{t^{1-\alpha}}t^{-\alpha - \beta }dt \right) ,
\end{align*}
since $\alpha < 1$. Thus,
\[
\sum_{k=E(n/2)+1}^{n}e^{k^{1-\alpha}}k^{-\alpha - \beta} = O \left( e^{n^{1-\alpha}n^{-\beta}} \right) .
\]
As a conclusion, we have
\begin{align*}
\sum_{k=E(n/2)+1}^{n}e^{-\sum_{j=k}^{n}u_{j}}u_{k}v_{k} & = O \left( e^{-(n+1)^{1-\alpha} + n^{1-\alpha}}v_{n} \right)  \\
& = O \left( v_{n} \right) .
\end{align*}

\end{proof}

\end{appendix}

\def\cprime{$'$}

\end{document}